\documentclass[12pt]{article}

\usepackage{amsfonts,pb-diagram}
\usepackage{amssymb}
\usepackage{amsmath,amsthm}
\usepackage{latexsym}
\usepackage{amsbsy}
\usepackage{amscd}
\usepackage[all]{xy}
\CompileMatrices

\newtheorem{teorema}{Theorem}[section]

\newtheorem{sublema}[teorema]{Sublemma}
\newtheorem{lema}[teorema]{Lemma}
\newtheorem{corolario}[teorema]{Corollary}
\newtheorem{ejemplo}[teorema]{Example}

\newtheorem{definicion}[teorema]{Definition}
\newtheorem{nota}[teorema]{Remark}

\def\bpm{B(\P^m,2)}
\def\fpm{F(\P^m,2)}

\def\TC{\mathrm{TC}}
\def\F2{\mathbb{F}_2}
\def\Z2{\mathbb{Z}_2}
\def\Sq{\mathrm{Sq}}
\def\sq1{\mathrm{Sq}^1}

\def\emb{\mathrm{Emb}}
\def\Imm{\mathrm{Imm}}

\def\P{\mathrm{P}}

\title{The integral cohomology of configuration spaces of pairs of points in real projective spaces}
\author{Carlos Dom\'{\i}nguez\hspace{.1mm}\footnote{Supported by Conacyt Ph.D.~scholarship number 162645.}\hspace{.5mm}, Jes\'us Gonz\'alez\hspace{.1mm}\footnote{Partially supported 
by CONACYT Research Grant number 102783.}\hspace{.5mm}, and
Peter Landweber}
\date{\empty}

\begin{document}

\maketitle

\begin{abstract}
We compute the integral cohomology ring of configuration spaces of two points on a given real projective space. Apart from an integral class, the resulting ring is a quotient of the known integral cohomology of the dihedral group of order~8 (in the case of unordered configurations, thus has only 2- and 4-torsion) or of the elementary abelian 2-group of rank 2 (in the case of ordered configurations, thus has only 2-torsion). As an application, we complete the computation of the symmetric topological complexity of real projective spaces $\P^{2^i+\delta}$ with \hspace{.4mm}$i\geq0$\hspace{.4mm} and 
\hspace{.4mm}$0\leq\delta\leq2$.
\end{abstract}



\section{A brief outline of the paper}
We compute the  
integral cohomology rings of 
$F(\P^m,2)$ and $B(\P^m,2)$,
the configuration spaces 
of two distinct points,
ordered and unordered respectively,
in the $m$-dimensional real projective 
space $\P^m$. Our explicit results are presented in Theorems~\ref{chf2pm}--\ref{decirbasicos} for $F(\P^m,2)$, and in Theorems~\ref{chb2pm}--\ref{baseB} for $B(\P^m,2)$. Proofs are given in Section~\ref{HF} for $F(\P^m,2)$, and in Sections~\ref{ahoralaB} and~\ref{HBring} for $B(\P^m,2)$.

\medskip
These rather technical calculations arose from a study of the symmetric topological complexity ($\TC^S$) of $\P^m$, and its relation to the embedding dimension of this manifold (Section~\ref{tcsapplicacion} recalls the basics of this relationship). In particular, our cohomological calculations allow us to complete the determination, started in \cite{symmotion}, of $\TC^S(\P^{2^i+\delta})$ for $i\geq0$ and $0\leq\delta\leq2$. The explicit new $\TC^S$-result is given in Theorem~\ref{STC}; the global $\TC^S$-picture for these projective spaces is summarized in~(\ref{fam1})--(\ref{laexcepcion}).

\section{Cohomology rings}\label{descriptionofcohomologias}
Unless indicated otherwise, the notation $H^*(X)$ refers to the integral cohomology ring of a space $X$ where a simple system of local coefficients is used. The degree of a cohomology class is explicitly indicated by means of an subscript: $c_k\in H^k(X)$. The cyclic group with $2^e$ elements is denoted by $\mathbb{Z}_{2^e}$. In the case $e=1$ we also use the notation $\F2$ if the field structure is to be noted. It will be convenient to use the notation $\langle k\rangle$ for the elementary abelian 2-group of rank $k$, and write $\{k\}$ as a shorthand for $\langle k\rangle\oplus\mathbb{Z}_4$.

\medskip
Recall that the ring $H^*(\P^\infty\times\P^\infty)$ is generated over the integers by three classes $x_2$, $y_2$, and $z_3$ subject only to the four relations
\begin{equation}\label{relacionesenteras}
2x_2=0,\;\;\;\;2y_2=0,\;\;\;\;2z_3=0,\;\;\;\mbox{and}\;\;\;\;z_3^2+x_2y_2(x_2+y_2)=0.
\end{equation}
The mod 2 reduction map $\rho\colon H^*(\P^\infty\times\P^\infty)\to H^*(\P^\infty\times\P^\infty;\F2)$ is characterized by 
\begin{equation}\label{losvaloresdelareduction}
\rho(x_2)=x_1^2, \;\;\;\;\rho(y_2)=
y_1^2, \;\;\;\;\mbox{and} \;\;\;\;
\rho(z_3)=x_1y_1(x_1+y_1). 
\end{equation}
Here $x_1, y_1\in H^*(\P^\infty\times\P^\infty;\F2)=H^*(\P^\infty;\F2)\otimes H^*(\P^\infty;\F2)$ are given by $x_1=z_1\otimes1$ and $y_1=1\otimes z_1$ where $z_1\in H^1(\P^\infty;\F2)$ is the generator (cf.~\cite[Example~3E.5]{hatcher}). We also use the notation $x_2$, $y_2$, and $z_3$ (with integral coefficients), as well as $x_1$ and $y_1$ (with mod 2 coefficients) for the images of the corresponding classes under the homomorphism of cohomology rings induced by the obvious inclusion
\begin{equation}\label{obvious-inclusion}
\alpha\colon F(\P^m,2)\hookrightarrow\P^\infty\times\P^\infty.
\end{equation}

\begin{teorema}\label{chf2pm} 
Let $m=2t+\delta$, $\delta\in\{0,1\}$. The following relations hold in $H^*(F(\P^m,2))\hspace{-1mm}:$
\begin{equation}\label{comunes}
x_2^{t+1}=0,\;\;\;\;y^{t+1}_2=0,\;\;\;\mbox{and}\;\;\sum _{i,j\geq0,\;i+j=t}\hspace{-2mm}x_2^i y_2^jz_3=0.
\end{equation}
\begin{itemize}
\item[{\em(a)}] If $\delta=0$, the integral cohomology ring $H^*(F(\P^m, 2))$ is generated by $x_2$, $y_2$, $z_3$, and a class $w_{2m-1}$ subject only to the relations~\emph{(\ref{relacionesenteras}), (\ref{comunes}),} and
\begin{equation}\label{lasrelsevens}
x_2^t y_2^t =0,\;\;\;\;\sum x^i_2y^j_2z_3=0,\;\;\;\mbox{and}\;\;\;\;w_{2m-1}\mu=0,
\end{equation}
for $\mu\in\{x_2,y_2,z_3,w_{2m-1}\}$, where the sum in~\emph{(\ref{lasrelsevens})} runs over $i,j\geq 0$ with $i+j=t-1$.
\item[{\em(b)}] If $\delta=1$, the integral cohomology ring $H^*(F(\P^m, 2))$ is generated by $x_2$, $y_2$, $z_3$, and a class $w_m$ subject only to the relations~\emph{(\ref{relacionesenteras}), (\ref{comunes}),} and
\begin{equation}\label{lasrelsoddss}
w_my_2+x_2^tz_3=0\;\;\;\mbox{and}\;\;\;\;w_m\mu=0,\;\;\mbox{for}\;\;\mu\in\{x_2,z_3,w_m\}.
\end{equation}
\end{itemize}
\end{teorema}

Note that the ($x_2$ vs.~$y_2$)-symmetry in the presentation for $H^*(F(\P^{2t},2))$ no longer holds in~(\ref{lasrelsoddss}). Although this is an intrinsic phenomenon for $m\equiv3\bmod4$, the asymmetry is only apparent for $m\equiv1\bmod4$: in terms of the torsion-free generator $w'_{4\ell+1}=w_{4\ell+1}+z_3(x_2^{2\ell-1}+x_2^{2\ell-2}y_2+\cdots+x_2^\ell y_2^{\ell-1})$,~(\ref{lasrelsoddss}) is replaced by the ($x_2$ vs.~$y_2$)-symmetric relations $w'_{4\ell+1}x_2=(x_2^{2\ell}+\cdots+x_2^{\ell+1}y_2^{\ell-1})z_3$, $w'_{4\ell+1}y_2=(y_2^{2\ell}+\cdots+y_2^{\ell+1}x_2^{\ell-1})z_3$, $w'_{4\ell+1}z_3=x_2^{\ell+1}y_2^{\ell+1}$, and $(w'_{4\ell + 1})^2 = 0$.

\medskip
The relations listed in Theorem~\ref{chf2pm} are minimal for $m\geq3$, and lead to explicit descriptions of cohomology groups (Theorem~\ref{descripcionordenada} next) and $\F2$-bases for torsion subgroups (Theorem~\ref{decirbasicos} following).

\begin{teorema}\label{descripcionordenada}
For $\,t\geq1$,

\vspace{-4mm}$$
H^i(F(\P^{2t},2)) = \begin{cases}
\mathbb{Z}, & i=0\mbox{ ~or~ }i=4t-1;\\
\left\langle\frac{i}2+1\right\rangle, & i\mbox{ \hspace{.3mm}even,~ }1\leq i\leq2t;\\
\left\langle\frac{i-1}2\right\rangle, & i\mbox{ \hspace{.3mm}odd,~ }1\leq i\leq2t;\\
\left\langle2t+1-\frac{i}2\right\rangle, & i\mbox{ \hspace{.3mm}even,~ }
2t<i<4t-1;\rule{6mm}{0mm}\\
\left\langle2t-\frac{i+1}2\right\rangle, & i\mbox{ \hspace{.3mm}odd,~ }2t<i<4t-1;\\
0, & \mbox{otherwise}.
\end{cases}
$$
For $\,t\geq0$,

\vspace{-5mm}$$
H^i(F(\P^{2t+1},2)) = \begin{cases}
\mathbb{Z}, & i=0;\\
\left\langle\frac{i}2+1\right\rangle, & i\mbox{ \hspace{.3mm}even,~ }1\leq i\leq2t;\\
\left\langle\frac{i-1}2\right\rangle, & i\mbox{ \hspace{.3mm}odd,~ }1\leq i\leq2t;\\
\mathbb{Z}\oplus\langle t\rangle, & i=2t+1;\\
\left\langle2t+1-\frac{i}2\right\rangle, & i\mbox{ \hspace{.3mm}even,~ }2t+1<i\leq4t+1;\\
\left\langle2t+1-\frac{i-1}2\right\rangle, & i\mbox{ \hspace{.3mm}odd,~ }2t+1<i\leq4t+1;\\
0, & \mbox{otherwise}.
\end{cases}
$$
\end{teorema}

\begin{teorema}\label{decirbasicos}
Let $m=2t+\delta$ with $\delta\in\{0,1\}$. A graded $\,\F2$-basis for the torsion subgroups of $H^*(F(\P^m,2))$ can be chosen as follows: In even dimensions the basis consists of the monomials $x_2^iy_2^j$ with $0\leq i,j\leq t$, $(i,j)\neq(0,0)$ and, if $\,\delta=0$, $(i,j)\neq(t,t)$. In odd dimensions the basis consists of monomials $x_2^iy_2^jz_3$ with $0\leq i\leq t-1+\delta$ and $0\leq j\leq t-2+\delta$.
\end{teorema}

The following is a straightforward consequence of the three results above.

\begin{corolario}\label{consecuenciasF}
The map induced in integral cohomology by~\emph{(\ref{obvious-inclusion})}:
\emph{\begin{enumerate}
\vspace{-2mm}
\item\label{factepiF}\emph{surjects in positive dimensions onto the torsion subgroups of $H^*(F(\P^m,2));$}
\vspace{-3.2mm}
\item\label{factcokerF}\emph{has cokernel generated by $w_m$ (when $m$ is odd) and $w_{2m-1}$ (when $m$ is even)\emph{;}}
\vspace{-3.2mm}
\item\label{factkerF}\emph{is injective in dimensions at most $m$.}
\end{enumerate}}
\end{corolario}

Corollary~\ref{consecuenciasF}.\ref{factkerF} can be stated in more precise terms: Ker$(\alpha^*)$ is the ideal of $H^*(\P^\infty\times\P^\infty)$ generated by the right-hand-side terms of the equations in~(\ref{comunes})--(\ref{lasrelsoddss}). Earlier versions of this paper (available as~\cite{v1}) interpret the latter fact in terms of Fadell-Husseini's index theory. The proofs of Theorems~\ref{chf2pm}--\ref{decirbasicos} rely on first establishing the first two assertions of Corollary~\ref{consecuenciasF} through a Bockstein spectral sequence argument.

\medskip
Next we focus on $B(\P^m,2)$. Recall the following three facts about the dihedral group $D_8$ of order 8 (see for instance~\cite{handeltohoku}). The ring $H^*(D_8)$ is generated over the integers by four classes $a_2$, $b_2$, $c_3$, and $d_4$ subject only to the six relations
\begin{equation}\label{relacionesenterasenD8}
2a_2=0,\;\;\;\;2b_2=0,\;\;\;\;2c_3=0,\;\;\;\;4d_3=0,\;\;\;\;b_2^2+a_2b_2=0,\;\;\;\mbox{and}\;\;\;\;c_3^2+a_2d_4=0.
\end{equation}
The $\F2$-algebra $H^*(D_8;\F2)$ is generated by three classes $u_1$, $v_1$, $w_2$ subject only to
\begin{equation}\label{relacionesmod2enD8}
u_1^2=u_1v_1.
\end{equation}
The mod 2 reduction map $\rho\colon H^*(D_8)\to H^*(D_8;\F2)$ is characterized by 
\begin{equation}\label{losvaloresdelareductionB}
\rho(a_2)=v_1^2,\;\;\;\;\rho(b_2)=u_1v_1,\;\;\;\;\rho(c_3)=v_1w_2,\,\;\;\;\mbox{and}\;\;\;\;
\rho(d_4)=w_2^2. 
\end{equation}
We also use the notation $a_2$, $b_2$, $c_3$, and $d_4$ (with integral coefficients), as well as $u_1$, $v_1$, and $w_2$ (with mod 2 coefficients) for the images of the corresponding classes under the map
\begin{equation}\label{obvious-inclusion-B}
\beta\colon B(\P^m,2)\to BD_8
\end{equation}
that classifies the following action (cf.~\cite[Proposition~2.6]{handel68}):

\begin{definicion}\label{inicio1Handel}{\em
In the usual wreath product extension $1\to\Z2\times\Z2\to D_8\to\Z2\to1$, let $\rho_1,\rho_2\in D_8$ be the obvious generators of the normal subgroup $\Z2\times\Z2$, and let (the class of) $\rho\in D_8$ generate the quotient group $\Z2$ so that, via conjugation, $\rho$ switches $\rho_1$ and $\rho_2$. $D_8$ acts freely on the Stiefel manifold $V_{m+1,2}$ of orthonormal $2$-frames in $\mathbb{R}^{m+1}$ by setting $\,\rho(v_1,v_2)=(v_2,v_1)$, $\,\rho_1(v_1,v_2)=(-v_1,v_2)$ and $\,\rho_2(v_1,v_2)=(v_1,-v_2)$, so that the orbit space $V_{m+1,2}/D_8$ is contained in $B(\P^m,2)$ as a strong deformation retract.
}\end{definicion}

\begin{teorema}\label{chb2pm} 
Let $m=2t+\delta$, $\delta\in\{0,1\}$ and, for $r\geq0$, consider the elements $$\sigma_{2r}=\!\!\!\!\sum_{\mbox{\scriptsize$\begin{array}{c}i,j\geq0\\i+2j=r\end{array}$}}\!\!\!\!\!\binom{i+j}{j}\hspace{.3mm}a_2^id_4^j\qquad\mbox{and}\qquad \iota_{2r}=\begin{cases}2d_4^{\frac{r}2},&\mbox{if $\,r$ is even;}\\0,&\mbox{if $\,r$ is odd;}
\end{cases}$$ in $H^*(B(\P^m,2))$.
The following relations hold in $H^*(B(\P^m,2))\hspace{-1mm}:$
\begin{equation}\label{comunesB}
a_2\sigma_{2t}=0,\;\;\;\;b_2\sigma_{2t}+\iota_{2t+2}=0,\;\;\;\mbox{and}\;\;\;\;c_3\sigma_{2t}=0.
\end{equation}
\begin{itemize}
\item[{\em(a)}] If $\delta=0$, the integral cohomology ring $H^*(B(\P^m, 2))$ is generated by $a_2$, $b_2$, $c_3$, $d_4$, and a class $e_{2m-1}$ subject only to the relations~\emph{(\ref{relacionesenterasenD8}), (\ref{comunesB}),} and \begin{equation}\label{lasrelsevensB}
c_3\sigma_{2t-2}=0,\;\;\;\;b_2d_4\sigma_{2t-2}+\iota_{2t+4}=0,\;\;\;\;d_4^t=0,\;\;\;\mbox{and}\;\;\;\;e_{2m-1}\mu=0,
\end{equation}
for $\mu\in\{a_2,b_2,c_3,,d_4,e_{2m-1}\}$.
\item[{\em(b)}] If $\delta=1$, the integral cohomology ring $H^*(B(\P^m, 2))$ is generated by $a_2$, $b_2$, $c_3$, $d_4$, and a class $e_m$ subject only to the relations~\emph{(\ref{relacionesenterasenD8}), (\ref{comunesB}),}
\begin{equation}\label{lasrelsoddssB}
a_2\sigma_{2t+2}=0,\;\;b_2\sigma_{2t+2}+\iota_{2t+4}=0,\;\;c_3\sigma_{2t+2}=0,\;\;d_4^{\,t+1}=0,
\end{equation}
\begin{equation}\label{lasrelsoddssBenteras}
e_m^2=0,\;\;\mu e_m=\kappa b_2^{\kappa}c_3d_4^{\,\ell},\;\;c_3e_m=\eta d_4^{\,\ell+1},\;\;\mbox{and}\;\;\,d_4e_m=\sum_{i=1}^{\ell}\binom{t-i}{i-1}a_2^{t-2i}b_2c_3d_4^{\,i}.
\end{equation}
Here $\mu\in\{a_2,b_2\}$, $t=2\ell+\kappa$ with $\kappa\in\{0,1\}$, and $\eta=b_2$ if $\kappa=1$, whereas $\eta=2$ if $\kappa=0$, except perhaps for $m=5$.
\end{itemize}
\end{teorema}

For $m=5$, it is natural to expect $\eta=2$ in the product $c_3e_5$ appearing in~(\ref{lasrelsoddssBenteras}). Our methods assure, in any case, $\eta\in\{0,2\}$. For $m=3$, the third relation in~(\ref{lasrelsoddssBenteras}) gives $c_3e_3=b_2d_4$, a trivial element in view of Theorem~\ref{aditivoB} below---more explicitly, one can use~Lemma~\ref{reltionskdanbaseB}.\ref{nadauxiliar} in the final section of the paper. Except for the latter situation, the right hand side of each relation in~(\ref{lasrelsoddssBenteras}) is in `reduced' form, as follows from Theorem~\ref{baseB} below. In fact, the relations listed in Theorem~\ref{chb2pm} are minimal for $m\geq3$, and lead to explicit descriptions of  cohomology groups (Theorem~\ref{aditivoB} next) and minimal generators for torsion subgroups (Theorem~\ref{baseB} following).

\begin{teorema}\label{aditivoB}
Let $0\leq b\leq3$. For $t\geq1$,

$$
H^{4a+b}(B(\P^{2t},2)) = \begin{cases}
\mathbb{Z}, & 4a+b=0 \mbox{ ~or~ } 4a+b=4t-1;\\
\{2a\}, & b=0<a,\;\,4a+b\leq2t;\\
\left\langle2a\right\rangle, & b=1,\;\,4a+b\leq2t;\\
\left\langle2a+2\right\rangle, & b=2,\;\,4a+b\leq2t;\\
\left\langle2a+1\right\rangle, & b=3,\;\,4a+b\leq2t;\\
\{2t-2a\}, & b=0,\;\,2t<4a+b<4t-1;\\
\langle2t-2a-1\rangle, & b=1,\;\,2t<4a+b<4t-1;\\
\langle2t-2a\rangle, & b=2,\;\,2t<4a+b<4t-1;\hspace{11mm}\\
\langle2t-2a-2\rangle, & b=3,\;\,2t<4a+b<4t-1;\\
0, & \mbox{otherwise}.
\end{cases}
$$

\noindent For $t\geq0$,

\vspace{-8mm}$$
H^{4a+b}(B(\P^{2t+1},2)) = \begin{cases}
\mathbb{Z}, & 4a+b=0;\\
\{2a\}, & b=0<a,\;\,4a+b\leq2t;\\
\left\langle2a\right\rangle, & b=1,\;\,4a+b\leq2t;\\
\left\langle2a+2\right\rangle, & b=2,\;\,4a+b\leq2t;\\
\left\langle2a+1\right\rangle, & b=3,\;\,4a+b\leq2t;\\
\mathbb{Z}\oplus\langle t\rangle, & 4a+b=2t+1;\\
\{2t-2a\}, & b=0,\;\,2t+1<4a+b\leq4t+1;\\
\langle2t-2a+1\rangle, & b=1,\;\,2t+1<4a+b\leq4t+1;\\
\langle2t-2a\rangle, & b\in\{2,3\},\;\,2t+1<4a+b\leq4t+1;\\
0, & \mbox{otherwise}.
\end{cases}
$$
\end{teorema}

\begin{teorema}\label{baseB}
Let $m=2t+\delta$ with $\delta\in\{0,1\}$. A minimal set of generators for the torsion subgroups of $H^*(B(\P^m,2))$ is given by the monomials
\begin{equation}\label{gendoresenterosB}
\mbox{$a_2^ib_2^\varepsilon d_4^j\,$ (in even dimensions) and $\;a_2^ib_2^\varepsilon c_3d_4^j\,$ (in odd dimensions)}
\end{equation}
where $\,\varepsilon\in\{0,1\}$, $\,i,j\geq0$, $\;j\leq t+\delta-1$, and 
\begin{itemize}
\item $1\leq i+j+\varepsilon\leq t\,$ in even dimensions;
\vspace{-3mm}
\item $i+j+1<t+\delta\,$ in odd dimensions (note that this condition is independent of $\varepsilon$).
\end{itemize}
\end{teorema}

The following is a straightforward consequence of the last three results.

\begin{corolario}\label{suprayeccionB}
The map induced in integral cohomology by~\emph{(\ref{obvious-inclusion-B}):}
\emph{\begin{enumerate}
\vspace{-2mm}
\item\label{factepiB}\emph{surjects in positive dimensions onto the torsion subgroups of $H^*(B(\P^m,2));$}
\vspace{-3.2mm}
\item\label{factcokerB}\emph{has cokernel generated by $e_m$ (when $m$ is odd) and $e_{2m-1}$ (when $m$ is even)\emph{;}}
\vspace{-3.2mm}
\item\label{factkerB}\emph{is injective in dimensions at most $m$.}
\end{enumerate}}
\end{corolario}

Note that $\bpm$ and $\fpm$ become homology spheres after inverting 2. Such a fact holds integrally in the case of $B(\P^1,2)$ and $F(\P^1,2)$. Indeed, there are well-known homotopy equivalences
\begin{equation}\label{topologicalfacts}
F(\P^1,2)\simeq S^1\simeq B(\P^1,2)
\end{equation}
(cf.~\cite[Example~2.2]{SadokCohenFest}). Since our descriptions of the integral cohomologies of $F(\P^1,2)$ and $B(\P^1,2)$ are compatible with~(\ref{topologicalfacts}), we will assume $m>1$ in Sections~\ref{HF}--\ref{HBring}. The case of $\P^2$ is the only further situation where the ring structure is trivial integrally---with $z_3=0$ for $H^*(F(\P^2,2))$, and $c_3=d_4=0$ for $H^*(B(\P^2,2))$.

\medskip
Our results can be coupled with the Universal Coefficient Theorem, expressing homology in terms of cohomology, to give an explicit description of the integral homology groups of $F(\P^m,2)$ and $B(\P^m,2)$. Likewise, in combination with Poincar\'e duality (in its not necessarily orientable version, cf.~\cite[Theorem~3H.6]{hatcher} or~\cite[Theorem~4.51]{ranicki}), our results lead to explicit descriptions of the $w_1$-twisted homology and cohomology groups of $F(\P^m,2)$ and $B(\P^m,2)$. Details are given in~\cite{v1}. 

\medskip
Theorems~\ref{chb2pm}--\ref{baseB} fully extend the calculations of $H^i(\bpm)$ given in~\cite{bausum} for $i$ close to the top cohomological dimension $2m-1$. Bausum's work led to a description of the sets of isotopy classes of smooth embeddings of $\P^m$ in $\mathbb{R}^{2m-e}$ for low values of $e$ (as low as $e\leq2$). Similar results were obtained by Larmore and Rigdon (note the implicit hypothesis $m>3$ in~\cite[Section~4]{LR})\footnote{We thank Sadok Kallel for pointing out the results in~\cite{bausum} and~\cite{LR}.}. Instead, our $\TC^S$-application follows the method outlined in~\cite{taylor}.

\medskip
Our original (additive) approach to $H^*(B(\P^m,2))$ and $H^*(F(\P^m,2))$ was based on the Cartan-Leray spectral sequence of the $D_8$-action in Definition~\ref{inicio1Handel}, and of the restricted action to the normal subgroup $\Z2\times\Z2$. This eventually gave the ring structures (the case of $B(\P^m,2)$ is part of the Ph.D.~thesis of the first author). The Bockstein spectral sequence approach in this paper was suggested by the referee, and leads to condensed proofs---despite that we have spent quite some space giving concrete details and explicit examples of our technical arguments. However, the current gain in brevity sacrifices the geometric motivation in~\cite{v1}, replacing it by a highly technical bookkeeping of cohomology groups through very explicit generators and relations. Thus, it is worth keeping in mind that~\cite{tesiscarlos,v1} offer (and make use of) a more geometric understanding of the central role played by $D_8$. In particular,~\cite{tesiscarlos} explains how the relations~(\ref{comunes})--(\ref{lasrelsoddss}) and~(\ref{comunesB})--(\ref{lasrelsoddssBenteras}) arise naturally as key differentials in the relevant Cartan-Leray spectral sequences.

\section{Symmetric topological complexity}\label{tcsapplicacion}
We now apply the cohomological information in the previous section to the problem of computing the symmetric topological complexity ($\TC^S$) of real projective spaces. As a motivation, we begin with a description of the relationship between $\TC^S$ and the embedding dimension of these manifolds. The relevant references for the facts in the next paragraph are~\cite{taylor, symmotion}, and we assume familiarity with the notation in those papers.

\medskip
Consider the homotopy class
\begin{equation}\label{classify}
B(\P^m,2)\stackrel{u}\longrightarrow\P^\infty
\end{equation}
classifying the obvious double cover $F(\P^m,2)\to B(\P^m,2)$. 
With the seven possible exceptions\footnote{Remark~\ref{optimalidad} below observes that we can now rule out the first of these potential exceptions.} of $m$ 
explicitly described in~\cite[Equation~(8)]{taylor}, $\emb(\P^m)$---the dimension of the smallest Euclidean space in which $\P^m$ can be smoothly embedded---is 
characterized as the smallest integer $e(m)$ such that the map in~(\ref{classify}) can be homotopy compressed into $\P^{e(m)-1}$. On the other hand, the main result in~\cite{symmotion} asserts that, without restriction on~$m$, $e(m)$ agrees with Farber-Grant's symmetric topological complexity\footnote{We follow the convention in~\cite{taylor} of using the reduced version of $\TC^S$, i.e.~we choose to normalize the Schwarz genus of a product fibration $F\times B\to B$ to be $0$---not $1$.} of $\P^m$, $\TC^S(\P^m)$. The latter is an invariant proposed in~\cite{FGsymm} to measure the inherent topological difficulties in the problem of finding ``efficient'' motion algorithms in robotics. Consequently, potentially new nonembedding results for $\P^m$---as well as inherent difficulties in the problem of planning symmetric motion in $\P^m$---could be deduced from the simple observation that, for a generalized cohomology theory $h^*$ with products, every class $z\in h^*(\P^\infty)$ must satisfy
\begin{equation}\label{cuppower}
u^*(z)^{e(m)}=0.
\end{equation}
The idea actually goes back at least as far as~\cite{handel68}, where mod 2 coefficients (and obstruction theory) are used. But the $\mathbb{Z}_4$ groups appearing in Theorem~\ref{aditivoB} carry finer information not yet explored\footnote{Compare with the situation in~\cite{electron} where the topological Borsuk problem for $\mathbb{R}^3$ is studied via Fadell-Husseini index theory.}. For instance, the strategy using integral coefficients has recently been exploited in~\cite{taylor} in order to compute $\TC^S(\mathrm{SO}(3))$---identifying it as the unique obstruction in Goodwillie's embedding Taylor tower for $\P^3$. The same idea now leads to:

\begin{teorema}\label{STC}
$\TC^S(\P^5)=\TC^S(\P^6)=9$.
\end{teorema}

Before proving this result, we compare it (in Remark~\ref{optimalidad} below) with known information (summarized in~\cite{tablas}) on $\emb(\P^m)$ for $m=5,6,7$, pausing to explain the way Theorem~\ref{STC} gives an exceptional situation to some general patterns of values of $\TC^S(\P^{m})$ (see~(\ref{fam1})--(\ref{laexcepcion})).

\begin{nota}\label{optimalidad}{\em
Since $\emb(\P^5)=9$ (\cite{hopf,mahowald}), the list 
in~\cite{taylor} of seven exceptional values of $m$ for which the equality 
$\emb(\P^m)=\TC^S(\P^m)$ {\it could\hspace{.5mm}} 
fail reduces now to $\{6,7,11,12,14,15\}$. Note that $6$ is the smallest $m$ for which $\emb(\P^m)$ is 
unknown: $\emb(\P^6)\in\{9,10,11\}$ is the best assertion known to date 
(\cite{ES, mahowald}). On the other hand, Theorem~\ref{STC} obviously implies 
$\TC^S(\P^7)\geq9$, improving by 1 the previously known best lower 
bound for $\TC^S(\P^7)$ noted in~\cite[Table~1]{taylor}. In fact, 
taking into account Rees' PL embedding $\P^7\subset
\mathbb{R}^{10}$ constructed in~\cite{rees}, the above considerations imply
that both $\TC^S(\P^7)$ and $\emb_{\mathrm{PL}}(\P^7)$ lie in $\{9,10\}$. This contrasts with the best known assertion about the smooth
embedding dimension of $\P^7$, namely $\emb(\P^7)\in\{9,10,11,12\}$
(\cite{Ha,M64}).
}\end{nota}

Except for three special cases (related to the Hopf invariant one problem), the reduced version of Farber's original (non-symmetric) topological complexity captures the immersion dimension of real projective spaces: As proved in~\cite{FTY}, the equality $\Imm(\P^m)=\TC(\P^m)$ holds for $m\neq1,3,7$. However, Remark~\ref{optimalidad} suggests that the equality $\emb(\P^m)=\TC^S(\P^m)$ could actually hold for every $m$, at least if $\emb$ is interpreted as {\it topological} embedding dimension. From such a perspective, it would be highly desirable to know whether $\P^6$ topologically embeds in $\mathbb{R}^9$. On the other hand, it does not seem likely that $\P^7$ could possibly embed in 
$\mathbb{R}^9$ (even topologically), and the techniques proving 
Theorem~\ref{STC} (using perhaps a cohomology theory better suited than 
singular cohomology) might allow us to formalize our intuition---we hope 
to come back to such a point elsewhere.

\medskip
Before getting into the main technical computation of this section, we set Theorem~\ref{STC} in context. The inequality 
\begin{equation}\label{inequalityX}
\TC^S(X)-\TC(X)\geq0
\end{equation}
is proved in~\cite[Corollary~9]{FGsymm} for any space $X$. It is optimal since, as proved in~\cite{symmotion},~(\ref{inequalityX}) becomes an equality if $X$ is, for instance, a complex projective space. However, as discussed in~\cite[Example~3.3]{symmotion}, there is no current indication that the left hand side in~(\ref{inequalityX}) should even be a bounded function of $m$ for $X=\P^m$. We discuss next the known situation (as updated by Theorem~\ref{STC}) for a few particular families of $m$. We use~\cite{tablas,FTY} as the main references for the known numerical values of $\TC(\P^m)$.

\medskip
To begin with, Example~3.3 in~\cite{symmotion} observes that
\begin{equation}\label{fam1}
\TC^S(\P^{2^i})-\TC(\P^{2^i})=1
\end{equation}
for any $i\geq0$ (the case $i=0$ was not mentioned in~\cite{symmotion}, but it is covered by the calculations in~\cite{F1,FGsymm}). Example~3.3 in~\cite{symmotion} also notes that
\begin{equation}\label{fam2}
\TC^S(\P^{2^i+1})-\TC(\P^{2^i+1})=2
\end{equation}
for any $i\geq3$; the corresponding result for $i=1,2$ is also true in view of~\cite{taylor} (for $i=1$) and Theorem~\ref{STC} (for $i=2$). Lastly, Example~3.3 in~\cite{symmotion} remarks that
\begin{equation}\label{laexcepcion}
\TC^S(\P^{2^i+2})-\TC(\P^{2^i+2})=1
\end{equation}
for any $i\geq4$. Now, while~(\ref{laexcepcion}) is also true for $i=3$ (as remarked in~\cite[Table~1]{taylor}), Theorem~\ref{STC} implies that, for $i=2$,~(\ref{laexcepcion}) must be replaced by $\TC^S(\P^{6})-\TC(\P^{6})=2$.

\medskip
Returning to this section's main focus (the proof of Theorem~\ref{STC}), we take advantage of the obvious inequality $e(m)\leq e(m+1)$ and of the fact that $e(6)\leq 9$---proved in~\cite[Corollary~11]{reesodd}---to reduce the proof of Theorem~\ref{STC} to proving the inequality $e(5)\geq9$. For this purpose, since the plan is to use integral cohomology, it will be simpler to replace~(\ref{cuppower}) by the observation that any cohomology class $z_d\in H^d(\P^\infty)$ with $d\geq e(m)$ must lie in the kernel of $u^*$. Thus, Theorem~\ref{STC} is a consequence of:

\begin{teorema}\label{altura}
For $m=5$, the homomorphism on integral cohomology induced by the map in~{\em(\ref{classify})} is monic in dimension $8$.
\end{teorema}

The proof of Theorem~\ref{altura} is based on Handel's observation (Lemma~\ref{laobsdehand} below) that~(\ref{classify}) factors through the classifying space of the dihedral group $D_8$ of order 8. 

\begin{lema}\label{laobsdehand}
The map in~{\em(\ref{classify})} corresponds to the pullback under~{\em(\ref{obvious-inclusion-B})} of the class $u_1$ appearing in~{\em(\ref{relacionesmod2enD8})}.
\end{lema}
\begin{proof}
This is proved in~\cite[Proposition~3.5]{handel68} under the extra hypothesis $m\geq 3$, but the restriction can be removed by naturality.
\end{proof}

Thus, the homotopy class in~(\ref{classify}) factors as $B(\P^m,2)\stackrel\beta\rightarrow BD_8\stackrel{q\,}\to\P^\infty$, where $q$ corresponds to the cohomology class $u_1\in H^1(D_8;\F2)$. Our last ingredient for the proof of Theorem~\ref{altura} is a description of the effect of $q$ in integral cohomology. With this in mind, we note that the group extension in Definition~\ref{inicio1Handel} gives a fibration $$\P^\infty\times\P^\infty\stackrel{\iota}\to BD_8\stackrel{q'}\to\P^\infty.$$ On the other hand, Handel's proof of~\cite[Proposition~3.5]{handel68} characterizes $u_1$ as the only nonzero element in $H^1(BD_8;\F2)$ 
mapping trivially under the fiber inclusion $\iota$. Thus, in fact $q=q'$.
In particular, the map induced  
by $q$ in 
integral cohomology 
can be computed 
in purely algebraic terms, using the projection in the group extension in Definition~\ref{inicio1Handel}.
Actually, since $H^*(\P^\infty)=\mathbb{Z}[z_2]\left/2z_2\right.$ 
where $z_2\in H^2(\P^\infty)=\mathbb{Z}_2$ is the generator, $q^*$ is determined by its 
value on $z_2$. A simple exercise using the Wall-Hamada resolution of
the trivial $D_8$-module $\mathbb{Z}$ (see for instance~\cite{handeltohoku})
shows that our generators in~(\ref{relacionesenterasenD8}) can be chosen\footnote{This depends on the user's choice of generators $x$ and $y$ for $D_8$ right at the beginning of~\cite{handeltohoku}.}
so that
\begin{equation}\label{detmncn}
q^*(z_2)=b_2.
\end{equation}
\begin{proof}[Proof of Theorem~{\em\ref{altura}}] In view of~(\ref{detmncn}) and Lemma~\ref{laobsdehand} we only need to check that $b_2^4\neq0$ in $H^*(B(\P^5,2))$---a straightforward task in view of our fine cohomological control of $B(\P^5,2)$:
\begin{eqnarray*}
b_2^4&=&a_2^3b_2\quad\qquad\hspace{-.2mm}\mbox{in view of the fifth relation in~(\ref{relacionesenterasenD8})}\\ &=&a_2b_2d_4\qquad\hspace{.1mm}\mbox{in view of the second relation in~(\ref{comunesB})}\\ &=&2d_4^2\quad\;\qquad\mbox{in view of Lemma~\ref{reltionskdanbaseB}.\ref{nadauxiliar} with $s=1$.}
\end{eqnarray*}
But $d_4^2$ is an element of order $4$ in view of Theorems~\ref{aditivoB} and~\ref{baseB}.
\end{proof}

\begin{nota}\label{noextns}{\em
The same method recovers the equation $\TC^S(\P^3)=5$, proved in~\cite[Theorem~1.4]{taylor}. It should be noted that the cup-power of $b_2\in H^*(B(\P^m,2))$---i.e.~the highest nontrivial cup power of this element---has been described for general $m$ in the Ph.D.~thesis~\cite{tesiscarlos} of the first author. Unfortunately, such a result gives no further information on $\emb(\P^m)$ or, for that matter, on $\TC^S(\P^m)$---this cup-power is just too low for $m\geq7$. This suggests the desirability of computing $h^*(\bpm)$ for other (richer) multiplicative cohomology theories. In such a generalized cohomology setting,~(\ref{cuppower}) could play, together with the concept of {\it weight of  a---generalized---cohomology class}, a more important role than in the current singular cohomology approach, cf.~\cite{ops}. We intend to eventually come back to these ideas.
}\end{nota}

\section{The cohomology ring $H^*(F(\P^m,2))$}\label{HF}

A quick look at the Cartan-Leray spectral sequence for the ($\mathbb{Z}_2\times\mathbb{Z}_2$)-action on $V_{m+1,2}$ in Definition~\ref{inicio1Handel} shows that $H^*(F(\P^m,2))$ has no odd torsion (cf.~\cite{v1}). So, in this section we compute these integral cohomology groups via a thorough study of the 2-primary Bockstein spectral sequence (BSS) of $F(\P^m,2)$.

\medskip\noindent
{\bf The first page of the BSS.} \hspace{.2mm} The following description of the ring $H^*(F(\P^m, 2);\F2)$ was first brought to the authors' attention by Frederick Cohen. Recall the cohomology classes $x_1$ and $y_1$ introduced in the sentence following~(\ref{losvaloresdelareduction}).

\begin{lema}\label{chf2pmmod2}
The map~\emph{(\ref{obvious-inclusion})} induces an epimorphism $H^*(\P^\infty\times\P^\infty;\F2)\to H^*(F(\P^m,2);\F2)$ of rings whose kernel is the ideal generated by the three elements $x_1^{m+1}$, $y_1^{m+1}$, and $\,\sum x_1^iy_1^j$, where the summation runs over $i,j\geq0$, $i+j=m$.
\end{lema}

\begin{proof} The kernel of the morphism induced by the inclusion $\P^m\times\P^m\hookrightarrow\P^\infty\times\P^\infty$ is generated by $x_1^{m+1}$ and $y_1^{m+1}$. The sum $\sum x_1^iy_1^j$ maps to the diagonal cohomology class in $\P^m\times\P^m$ in view of~\cite[Theorem~11.11]{MS}---which restricts to zero in $F(\P^m,2)$. So it suffices to check that the inclusion $F(\P^m,2)\hookrightarrow\P^m\times\P^m$ induces an epimorphism whose kernel is generated by the diagonal class. But~(see \cite[Section~11]{MS}) the map under consideration embeds into a long exact sequence
$$
\cdots\to H^{*-m}(P^m;\Z2)\to H^*(\P^m\times\P^m;\Z2)\to H^*(F(\P^m,2);\Z2)\to\cdots
$$
(written here in terms of the Thom isomorphism for the normal bundle of the diagonal inclusion $\P^m\hookrightarrow\P^m\times\P^m$). The desired conclusion follows from~\cite[Lemma~11.8]{MS} which shows that the map of degree $m$ in this long exact sequence is given by multiplication by the diagonal class $\sum_{i+j=m}x_1^iy_1^j$---a monomorphism in the current case.
\end{proof}

\noindent{\bf First order Bocksteins.} Lemma~\ref{chf2pmmod2} implies that the monomials
\begin{equation}\label{baseasimetrica}
x_1^iy_1^j\quad\mbox{with}\quad0\leq i\leq m,\;0\leq j\leq m-1
\end{equation}
form an $\F2$-basis for the initial page of the BSS. Consider the filtration\footnote{This filtration was suggested by the referee.} $0=F^3\subseteq F^2\subseteq F^1\subseteq F^0=H^*(F(\P^m,2);\F2)$ where $F^k$ is generated by the basis elements in~(\ref{baseasimetrica}) with:
\begin{itemize}
\item $j<m-1$ if $j$ is odd, for $k=1$;

\vspace{-3mm}
\item even $j$, for $k=2$.
\end{itemize}
(Note that $F^1=F^0$ if $m$ is odd.) The filtration is stable under the action of the first Bockstein $\Sq^1$, and we describe next the resulting ``auxiliary'' spectral sequence---converging to the second page of the BSS for $F(\P^m,2)$. In what follows, the reader should keep in mind that the derivation $\Sq^1$ is characterized by $\Sq^1a^{k}=k a^{k+1}$ for $a\in\{x_1,y_1\}$.

\medskip
An $\F2$-basis for the $\Sq^1$-cohomology of $F^2$ is given by (the classes of) $1,y_1^2,\ldots,y_1^{m-2+\delta}$ and, if $m$ is odd, $x_1^{m},x_1^my_1^2,\ldots,x_1^my_1^{m-1}$. Here $m=2t+\delta$ with $\delta\in\{0,1\}$---so that $t$ is as in Theorem~\ref{chf2pm}. Likewise, an $\F2$-basis for the $\Sq^1$-cohomology of $F^1/F^2$ is given by (the classes of) $y_1,y_1^3,\ldots,y_1^{m-3+\delta}$ and, if $m$ is odd, $x_1^{m}y_1, x_1^my_1^3,\ldots,x_1^my_1^{m-2}$. Lastly, we have $F^0/F^1=0$ if $m$ is odd, while for $m$ even an $\F2$-basis for the $\Sq^1$-cohomology of $F^0/F^1$ is given by (the class of) $x_1^my_1^{m-1}$. All these assertions are obvious, except for the last one which requires the following calculation in $H^*(F(\P^m,2);\F2)$: for even $i$ with $0\leq i\leq 2t-2=m-2$,
\begin{eqnarray}
\Sq^1(x_1^iy_1^{2t-1})&=&x_1^iy_1^{2t}\;\,=\,\;x_1^i\left(x_1^{2t}+x_1^{2t-1}y_1+\cdots+x_1y_1^{2t-1}\right)\nonumber\\&=&x_1^{2t}y_1^i+x_1^{2t-1}y_1^{i+1}+\cdots+x_1^{i+1}y_1^{2t-1}\nonumber\\&\equiv&x_1^{i+1}y_1^{2t-1}\pmod{F^1}.\nonumber
\end{eqnarray}
The above considerations give the first page of the auxiliary spectral sequence. Note that, besides~$1$, $x_1^my_1^{m-1}$ (for even $m$) and $x_1^m$ (for odd $m$) represent permanent cycles in the auxiliary spectral sequence, for Lemma~\ref{chf2pmmod2} gives in $H^*(F(\P^m,2);\F2)$
\begin{eqnarray*}
\Sq^1(x_1^my_1^{m-1})&=&x_1^my_1^m\;\,=\;\,0,\mbox{\ \ for even $m$;}\\\Sq^1x_1^m&=&x_1^{m+1}\hspace{3mm}=\;\,0,\mbox{\ \ for odd $m$.}
\end{eqnarray*}
All other classes in the auxiliary spectral sequence are wiped out by $d_1$-differentials since, again in $H^*(F(\P^m,2);\F2)$,
\begin{eqnarray*}
\Sq^1y_1^j&=&y_1^{j+1},\mbox{\hspace{8mm}for odd $j$, $\,1\leq j\leq m-3+\delta$;}\\\Sq^1(x_1^my_1^j)&=&x_1^my_1^{j+1},\mbox{\ \ for odd $j$, $\,1\leq j\leq m-2$ \hspace{1mm} (relevant if $m$ is odd).}
\end{eqnarray*}
Thus, the auxiliary spectral sequence collapses from its second page which, as noted above, gives the second page of the BSS for $F(\P^m,2)$. Further, the BSS collapses from its second page for dimensional reasons. 

\medskip\noindent{\bf Immediate consequences.} The BSS-analysis yields the following standard implications:
\begin{itemize}
\item[(a)] the torsion-free subgroups in $H^*(F(\P^m,2))$ are as described in Theorem~\ref{descripcionordenada}, with non-torsion positive-dimensional cohomology classes $w_{2m-1}$ (for even $m$) and $w_m$ (for odd $m$) having mod~2 reductions
\begin{equation}\label{lareducciondelasws}
\rho(w_{2m-1})=x_1^my_1^{m-1}\mbox{ \  \ and \  \ }\rho(w_m)=x_1^m;
\end{equation}
\item[(b)] the torsion subgroups inject, via the mod~2 reduction map, into $H^*(F(\P^m,2);\F2)$ with image that of the endomorphism
\begin{equation}\label{endosq1}
\Sq^1\colon H^*(F(\P^m,2);\F2)\to H^*(F(\P^m,2);\F2).
\end{equation}
\end{itemize}
This and Lemma~\ref{chf2pmmod2} imply the first two items in Corollary~\ref{consecuenciasF}, and lead (as indicated below) to the groups in Theorem~\ref{descripcionordenada}. Yet, the finer multiplicative description (Theorems~\ref{chf2pm} and~\ref{decirbasicos}) requires a slightly more careful bookkeeping for the resulting classes in the image of~(\ref{endosq1}). This is spelled out next in terms of the torsion elements $x_2,y_2,z_3\in H^*(F(\P^m,2))$ defined in the sentence containing~(\ref{obvious-inclusion}). We work directly with the basis elements in~(\ref{baseasimetrica}), keeping the notation $m=2t+\delta$, $\;\delta\in\{0,1\}$.

\medskip\noindent{\bf Additive counting.} Consider the following partition of the basis elements in~(\ref{baseasimetrica}):
\begin{eqnarray*}
\mathcal{P}_0&=&\{\,\mbox{basis elements in~(\ref{baseasimetrica}) for which $j$ is even and either $i$ is even or $i=m\,$}\};\\\mathcal{P}_1&=&\{\,\mbox{basis elements in~(\ref{baseasimetrica}) for which $i$ and $j$ have distinct parity}\,\}\,-\,\mathcal{P}_0;\\\mathcal{P}_2&=&\{\,\mbox{basis elements in~(\ref{baseasimetrica}) for which both $i$ and $j$ are odd}\,\}\,-\,(\mathcal{P}_0\cup \mathcal{P}_1).
\end{eqnarray*}
Elements in $\mathcal{P}_0$ can be ignored as they have trivial $\Sq^1$-image. A straightforward calculation shows that the set of $\Sq^1$-images of elements in $\mathcal{P}_1$ is formed by the basis elements 
\begin{equation}\label{bonche1}
\mbox{$\rho(x_2^ay_2^b)=x_1^{2a}y_1^{2b}\;\;$ with $\,\;0\leq a\leq t$, $\,0\leq b\leq t-1+\delta\;,$ and $\;(a,b)\neq(0,0)$}
\end{equation}
and, when $\delta=0$, by the (sum of basis) elements
\begin{eqnarray}
\rho(x_2^ay_2^t)&=&x_1^{2a}y_1^{2t}\;\,=\;\,x_1^{2a}(x_1^{2t}+x_1^{2t-1}y_1+\cdots+x_1y_1^{2t-1})\nonumber\\
&=&x_1^{2t}y_1^{2a}+x_1^{2t-1}y_1^{2a+1}+\cdots+x_1^{2a+1}y_1^{2t-1}\label{bonche2}
\end{eqnarray}
for $0\leq a\leq t-1$. Since the elements listed in~(\ref{bonche1}) and~(\ref{bonche2}) are linearly independent, this proves Theorem~\ref{decirbasicos} and, by a simple counting, Theorem~\ref{descripcionordenada}, both in even dimensions. Likewise, the set of $\Sq^1$-images of elements in $\mathcal{P}_2$ is formed by the linearly independent elements
\begin{equation}\label{bonche3}
\mbox{$\rho(x_2^ay_2^bz_3)=x_1^{2a+2}y_1^{2b+1}+x_1^{2a+1}y_1^{2b+2}\;$ with $\;0\leq a\leq t-1+\delta\hspace{.31mm}$ and $\,0\leq b\leq t-2+\delta$.}
\end{equation}
[Note that, when $\delta=0$, the previous assertion would seem to miss the $\Sq^1$-image of basis elements~(\ref{baseasimetrica}) of the form $x_1^{2a+1}y_1^{2t-1}$ with $0\leq a\leq t-1$. However, Lemma~\ref{chf2pmmod2} gives
\begin{eqnarray*}
\Sq^1(x_1^{2a+1}y_1^{2t-1})&=&x_1^{2a+2}y_1^{2t-1}+x_1^{2a+1}y_1^{2t}\\ &=&x_1^{2a+2}y_1^{2t-1}+x_1^{2a+1}\left(x_1^{2t}+x_1^{2t-1}y_1+\cdots+x_1y_1^{2t-1}\right)\\ &=&x_1^{2a+2}y_1^{2t-1}+x_1^{2t}y_1^{2a+1}+x_1^{2t-1}y_1^{2a+2}+\cdots+x_1^{2a+2}y_1^{2t-1}\\ &=&\left(x_1^{2t}y_1^{2a+1}+x_1^{2t-1}y_1^{2a+2}\right)+\cdots+\left(x_1^{2a+4}y_1^{2t-3}+x_1^{2a+3}y_1^{2t-2}\right),
\end{eqnarray*}
which is a linear combination---trivial if $a=t-1$---of the elements in~(\ref{bonche3}).] This proves Theorem~\ref{decirbasicos} and, again by a simple counting, Theorem~\ref{descripcionordenada}, now in odd dimensions. 

\medskip\noindent{\bf Ring structure.} It remains to prove Theorem~\ref{chf2pm}. The two relations $w_{2m-1}^2=0$ (for even $m$) and $w_m^2=0$ (for odd $m$) are forced for dimensional reasons in view of Theorem~\ref{descripcionordenada}. All other relations asserted in~(\ref{comunes})--(\ref{lasrelsoddss}) involve exclusively torsion summands and, in view of the assertion containing~(\ref{endosq1}), can be proved by reducing coefficients mod 2. Such a checking becomes a straightforward task (which is left to the reader) using Lemma~\ref{chf2pmmod2},~(\ref{losvaloresdelareduction}), and~(\ref{lareducciondelasws}). The crux of the matter, then, lies in showing (in the next paragraphs) that these relations give a complete ring presentation for $H^*(F(\P^m,2))$.

\medskip
For a positive integer $m$, consider the graded ring $\mathcal{R}_m=\mathbb{Z}[W,X,Y,Z]/I_m$ where $W,X,Y,Z$ are formal variables of respective degrees $2m-1,2,2,3$ for even $m$, and $m,2,2,3$ for odd $m$, and where $I_m$ is the ideal generated by polynomials $E=E(W,X,Y,Z)$ for which the corresponding element $e=E(w,x_2,y_2,z_3)\in H^*(F(\P^m,2))$ is one of the polynomial expressions on the left hand side of the relations listed in~(\ref{relacionesenteras}) and~(\ref{comunes})--(\ref{lasrelsoddss}). Here we have written $w$ for either $w_{2m-1}$ of $w_m$, according to whether $m$ is even of odd. For instance, for $m=2t$, three of the generators of $I_m$ are $Z^2+XY(X+Y)$, $W^2$, and $\sum X^iY^jZ$, where the summation runs over $i,j\geq0$ with $i+j=t-1$. Thus, we have an epimorphism of rings $\Phi_m\colon\mathcal{R}_m\to H^*(F(P^m,2))$, $\Phi_m(E)=e$. In order to show that this is a ring isomorphism (thus completing the proof of Theorem~\ref{chf2pm}) it suffices to check that 
\begin{equation}\label{elgoalfinalordenado}
\mbox{\it the $\F2$-basis in Theorem~\emph{\ref{decirbasicos}} comes from generators for the torsion groups of $\mathcal{R}_m$.}
\end{equation}
(Indeed, it is evident that $\Phi_m$ yields an isomorphism on the corresponding torsion-free subgroups, while the torsion subgroups of $\mathcal{R}_m$ are $\F2$-vector spaces.) 

\medskip
We start with the $\mathbb{Z}$-basis of monomials $W^iX^jY^kZ^\ell$, $i,j,k,\ell\geq0$, for $\mathbb{Z}[W,X,Y,Z]$, and use each of the generators in $I_m$ to rule out some of these basis elements---the (classes of the) remaining monomials will of course generate $\mathcal{R}_m$. In doing so, we can ignore all monomials $W^i$ with $i\geq0$, for we are focusing on torsion subgroups (so that the generators $2X$, $2Y$, and $2Z$ of $I_m$ are implicitly accounted for). 

\medskip
The generators $W^2$, $X^{t+1}$, $Y^{t+1}$, and $Z^2+XY(X+Y)$ of $I_m$ mean that our list of generating monomials reduces to 
\begin{equation}\label{baseinicial}
W^iX^jY^kZ^\ell,\quad 0\leq j,k\leq t,\quad0\leq i,\ell\leq1
\end{equation}
where, as usual, $m=2t+\delta$, $\delta\in\{0,1\}$. Further, the generators in $I_m$ which come from the relations in~(\ref{lasrelsoddss}) and the last relation in~(\ref{lasrelsevens}), i.e.~those involving $w$, imply that the restriction $0\leq i\leq1$ in~(\ref{baseinicial}) can in fact be strengthened to $i=0$. Thus, in even dimensions we are left with the generating monomials $X^jY^k$ with $0\leq j,k\leq t$, $(j,k)\neq(0,0)$, and, if $\delta=0$, $(j,k)\neq(t,t)$---in view of the generator $X^tY^t$ of $I_m$ for even $m$. This proves~(\ref{elgoalfinalordenado}) in even dimensions. On the other hand, in view of the generator $\sum X^jY^kZ$ of $I_m$ (the sum running over $j,k\geq0$ with $j+k=t$), in odd dimensions we are left with the generating monomials $X^jY^kZ$ with $0\leq j\leq t$ and $0\leq k\leq t-1$, which completes the proof of~(\ref{elgoalfinalordenado}) for odd $m$. Lastly, if $m$ is even, the first of the two generators of $I_m$ 
\begin{equation}\label{dosquequedan}
\sum_{\begin{smallmatrix}j+k=t-1\\j,k\geq0\end{smallmatrix}}\hspace{-1.7mm}X^jY^kZ\quad\;\mathrm{\,and}\,\quad\sum_{\begin{smallmatrix}j+k=t\\j,k\geq0\end{smallmatrix}}X^jY^kZ
\end{equation}
gives in $\mathcal{R}_m$ the relations $X^tZ=-(X^{t-1}Y+\cdots+XY^{t-1})Z=Y^tZ$, so that the second generator in~(\ref{dosquequedan}) can equivalently be replaced by $X^tZ$ (giving $Y^tZ\in I_m$ for free). Thus, this time in odd dimensions we are left with the generating monomials $X^jY^kZ$ with $0\leq j\leq t-1$ and $0\leq k\leq t-2$, which completes the proof of ~(\ref{elgoalfinalordenado}) for even $m$.

\section{The cohomology groups $H^*(B(\P^m,2))$}\label{ahoralaB}
As in the case of $H^*(F(\P^m,2))$ in the previous section, the Cartan-Leray spectral sequence for the $D_8$-action on $V_{m+1,2}$ in Definition~\ref{inicio1Handel} shows that $H^*(B(\P^m,2))$ has no odd torsion---alternatively, use the corresponding property for $F(\P^m,2)$, together with the transfer for the two-fold covering $F(\P^m,2)\to B(\P^m,2)$. Thus, in this section we make a thorough study of the 2-primary BSS of $B(\P^m,2)$ in order to deduce the integral cohomology groups of $B(\P^m,2)$.

\medskip\noindent
{\bf The first page of the BSS.} \hspace{.2mm}The following description of the ring $H^*(B(\P^m, 2);\F2)$ is proved in~\cite[Theorem~3.7]{handel68}. Recall the cohomology classes $u_1$, $v_1$, and $w_2$ introduced in the sentence containing~(\ref{relacionesmod2enD8}).

\begin{lema}\label{chb2pmmod2}
The map~\emph{(\ref{obvious-inclusion-B})} induces an epimorphism $\beta^*\colon H^*(BD_8;\F2)\to H^*(B(\P^m,2);\F2)$ of rings with kernel the ideal generated by the two elements
\begin{equation}\label{lassigmatoras}
\sum_{i\geq0}\binom{m-i}{i}v_1^{m-2i}w_2^i\qquad\mbox{and}\qquad\sum_{i\geq0}\binom{m+1-i}{i}v_1^{m+1-2i}w_2^i.
\end{equation}
\end{lema}

Settling a basis for the mod 2 cohomology of $\bpm$ requires a bit more work than in the case of the $F(\P^m,2)$-analogue~(\ref{baseasimetrica}).

\begin{lema}\label{arithauxi}
For $s=0,1,\ldots,m$, the elements $R_{m+s}=\sum_{i\geq0}\binom{m-s-i}{i}v_1^{m-s-2i}w_2^{s+i}$ vanish in $H^*(B(\P^m,2);\F2)$.
\end{lema}
\begin{proof}
The first relation in~(\ref{lassigmatoras}) gives $R_m=0$. The relation $R_{m+1}=0$ follows by adding the second element in~(\ref{lassigmatoras}) to the $v_1$-multiple of the first element in~(\ref{lassigmatoras})---and pulling back under $\beta$. The rest of the relations then follow inductively by noticing that $R_{m+s}=v_1R_{m+s-1}+w_2R_{m+s-2}$.
\end{proof}

\begin{corolario}\label{basemod2B}
An $\F2$-basis for $H^*(B(\P^m,2);\F2)$ is given by the monomials
\begin{equation}\label{basemod2Bdisp}
u_1^{\varepsilon}v_1^rw_2^s\,\mbox{ \ with \ $\,\varepsilon\leq1\,$ \ and \ $\,r+s<m$.}
\end{equation}
\end{corolario}
\begin{proof}
In view of~(\ref{relacionesmod2enD8}) and the relations $R_{m+s}=v_1^{m-s}w_2^s+\cdots=0$ in Lemma~\ref{arithauxi}, the indicated elements are additive generators. Linear independence follows from the next result, since the number of monomials in~(\ref{basemod2Bdisp}) matches the (graded-wise) $\F2$-dimension of $H^*(B(\P^m,2);\F2)$.
\end{proof}

\begin{sublema}\label{handel68B}
For any $m$, 

\vspace{-2.5mm}
$$H^i(B(\P^m,2);{\F2})=\begin{cases}\langle 
i+1\rangle,& 0\leq i\leq m-1;\\
\langle2m-i\rangle,& m\leq i\leq 2m-1;\\0,&\mbox{otherwise}.\end{cases}$$
\end{sublema}
\begin{proof}
The assertion for $i\geq2m$ follows from the fact that $B(\P^m,2)$ has the homotopy type of the closed ($2m-1$)-dimensional manifold $V_{m+1,2}/D_8$ (see Definition~\ref{inicio1Handel}). Poincar\'e duality then gives the assertion for $m\leq i\leq 2m-1$ as a consequence of that for $0\leq i\leq m-1$. Lastly, the assertion for $0\leq i\leq m-1$ follows from Lemma~\ref{chb2pmmod2} and the fact that $H^i(BD_8;{\F2})=\langle i+1\rangle$; indeed, the presentation~(\ref{relacionesmod2enD8}) implies that an $\F2$-basis for $H^*(BD_8;\F2)$ is given by all monomials $u_1^\varepsilon v_1^rw_2^s$ with $\varepsilon\leq1$ (this basis will be in force throughout the next considerations).
\end{proof}

\noindent{\bf The auxiliary spectral sequence.} The $\Sq^1$-action on $H^*(BD_8;\F2)$ is implicit in~\cite[Proposition~3.5]{handel68}: $\Sq^1(w_2)=v_1w_2$, and $\Sq^1(\xi_1)=\xi_1^2$ for $\xi_1\in\{u_1,v_1\}$. The Cartan formula then yields
\begin{equation}\label{cartanformula}
\Sq^1(u_1^\varepsilon v_1^rw_2^s)=(\varepsilon+r+s)u_1^\varepsilon v_1^{r+1}w_2^s,
\end{equation}
which also holds in $H^*(B(\P^m,2);\F2)$ by naturality. Consider the filtration $0=B^m\subseteq B^{m-1}\subseteq\cdots\subseteq B^1\subseteq B^0=H^*(B(\P^m,2);\F2)$ where $B^k$ is generated by the basis elements in~(\ref{basemod2Bdisp}) with $s\geq k$. Each $B^k$ is stable under the action of $\Sq^1$ in view of Lemma~\ref{arithauxi} and~(\ref{basemod2Bdisp}). We describe next the resulting ``auxiliary'' spectral sequence---converging to the second page of the BSS for $B(\P^m,2)$.

\medskip
For $k=0,\ldots,m-1$, a basis for $B^k/B^{k+1}$ is given by the monomials~(\ref{basemod2Bdisp}) with $s=k$ and, in these terms, the filtered $\Sq^1$-action takes the form
$$
\Sq^1(u_1^\varepsilon v_1^rw_2^k)=\begin{cases}(\varepsilon+r+k)u_1^\varepsilon v_1^{r+1}w_2^s,&r+k<m-1;\\0,&r+k=m-1,\end{cases}
$$
in view of~(\ref{cartanformula}) and Lemma~\ref{arithauxi}. Then, an $\F2$-basis for the cycles in the 0-th page of the auxiliary spectral sequence is given by the monomials in~(\ref{basemod2Bdisp}) for which either $r+s=m-1$ or $\varepsilon+r+s$ is even. Likewise, an $\F2$-basis for the corresponding boundaries is given by the monomials in~(\ref{basemod2Bdisp}) for which 
\begin{equation}\label{rankd0}
r>0\qquad\mbox{and}\qquad\varepsilon+r+s\equiv0\bmod2.
\end{equation}
Thus, an $\F2$-basis for the first page of the auxiliary spectral sequence is given by the monomials in~(\ref{basemod2Bdisp}) for which one of the following two conditions holds:
\begin{itemize}
\item[(a)] $r+s=m-1$, and either $r=0$ or $\varepsilon+r+s$ is odd.

\vspace{-3mm}
\item[(b)] $r=0$, $s<m-1$, and $\varepsilon+s$ is even.
\end{itemize}

The explicit elements of type (b) are $u_1^\varepsilon w_2^s$ for $0\leq s\leq m-2$ and $\varepsilon+s\equiv0\bmod2$, all of which are permament cycles in the auxiliary spectral sequence in view of~(\ref{cartanformula}). On the other hand, the explicit elements of type (a) are $u_1w_2^{m-1}$, $w_2^{m-1}$ and
\begin{equation}\label{lasd1sdiffsauxil}
u_1^\varepsilon v_1^{m-s-1}w_2^s\mbox{\ \ \ for \ }0\leq s\leq m-2\mbox{ \ and \ }\varepsilon+m\equiv0\bmod2.
\end{equation}
We show next that most of these $m+1$ elements are wiped out by $d_1$-differentials in the auxiliary spectral sequence, whereas the few $d_1$-cycles which are not $d_1$-boundaries are in fact permanent cycles. 

\medskip\noindent 
{\em Case $m$ even:} (Note that $\varepsilon=0$ in~(\ref{lasd1sdiffsauxil}).) The differentials
\begin{equation}\label{rankd1par}
d_1(v_1^{m-2i-1}w_2^{2i})=v_1^{m-2i-2}w_2^{2i+1},\quad0\leq i\leq\frac{m}{2}-1,
\end{equation}
hold since~(\ref{cartanformula}) and Lemma~\ref{arithauxi} give $\Sq^1(v_1^{m-2i-1}w_2^{2i})=v_1^{m-2i}w_2^{2i}\equiv v_1^{m-2i-2}w_2^{2i+1}$ mod~$B^{2i+2}$. On the other hand, the only element of type (a) not considered in the above $d_1$-differentials, namely $u_1w_2^{m-1}$, is in fact a permanent cycle in view of~(\ref{cartanformula}). 

\medskip\noindent 
{\em Case $m$ odd:} (Note that $\varepsilon=1$ in~(\ref{lasd1sdiffsauxil}).) The differentials
\begin{equation}\label{rankd1imp}
d_1(u_1v_1^{m-2i}w_2^{2i-1})=u_1v_1^{m-2i-1}w_2^{2i},\quad1\leq i\leq\frac{m-1}{2},
\end{equation}
hold since~(\ref{cartanformula}) and Lemma~\ref{arithauxi} give $\Sq^1(u_1v_1^{m-2i}w_2^{2i-1})=u_1v_1^{m-2i+1}w_2^{2i-1}\equiv u_1v_1^{m-2i-1}w_2^{2i}$ mod $B^{2i+1}$. On the other hand, the only two elements of type (a) not considered in the above $d_1$-differentials, namely $u_1v_1^{m-1}$ and $w_2^{m-1}$, are in fact permanent cycles. Indeed, the assertion is obvious from~(\ref{cartanformula}) in the case of $w_2^{m-1}$. For $u_1v_1^{m-1}$ use~(\ref{cartanformula}) and Lemma~\ref{arithauxi} to get 
$$
\Sq^1(u_1v_1^{m-1})=u_1v_1^m=\sum_{i\geq1}\binom{m-i}{i}u_1v_1^{m-2i}w_2^i,
$$
and note that $\binom{m-i}i$ is even if $i$ is odd, whereas $u_1v_1^{m-2i}w_2^i=\Sq^1(u_1v_1^{m-2i-1}w_2^i)$ if $i$ is even. So, the element
\begin{equation}\label{representantecorrecto}
u_1v_1^{m-1}+\sum_{i\geq1}\binom{m-2i}{2i}u_1v_1^{m-4i-1}w_2^{2i}
\end{equation}
is a permanent cycle in the auxiliary spectral sequence representing the same class as $u_1v_1^{m-1}$. 

\medskip\noindent{\bf Second order Bocksteins.} We have proved that an $\F2$-basis of the second page of the BSS of $B(\P^m,2)$ is represented by the monomials 
\begin{equation}\label{losmonomios}
u_1^\varepsilon w_2^s,\;\;\;0\leq s\leq m-1,\;\;\varepsilon+s\equiv0\;\bmod2,
\end{equation}
together with an extra basis element represented by~(\ref{representantecorrecto}) if $m$ is odd. Next we analyze the second Bockstein differentials in $B(\P^m,2)$ and, for this purpose, we begin by taking a look at the BSS of $BD_8$. Observe from~(\ref{cartanformula}) that an $\F2$-basis for the second page of the BSS for $BD_8$ is represented by the monomials $u_1^{\varepsilon}w_2^s$ with $\varepsilon+s\equiv0\bmod2$. Furthermore, the family of second Bockstein differentials
\begin{equation}\label{betas2}
\beta_2(u_1w_2^{2\ell-1})=w_2^{2\ell} \mbox{\ \ for \ }\ell\geq1
\end{equation}
follows from the fact that the only 4-torsion classes in $H^*(BD_8)$ come from the powers $d_4^{\hspace{.2mm}\ell}$, which are concentrated in positive dimensions congruent to zero modulo 4. In particular, the third page of the BSS for $BD_8$ is concentrated in degree 0, forcing its collapse from this page on. Now, the $\beta_2$-differentials in~(\ref{betas2}) pull back under the map in~(\ref{obvious-inclusion-B}) to yield a corresponding family of second Bockstein differentials in $B(\P^m,2)$; this wipes all of the monomials in~(\ref{losmonomios}), except for those with $s=0$ and, for even $m$, $s=m-1$. On the other hand, if $m\equiv1\bmod4$, the element in~(\ref{representantecorrecto}) has trivial $\beta_2$-differential because none of the elements in~(\ref{losmonomios}) lies in a dimension congruent to 2 mod 4. In any case, only two classes survive to the third page of the BSS of $B(\P^m,2)$: $1=u_1^0w_2^0\hspace{.5mm}$ and a class represented by 
\begin{itemize}
\item[(i)] $u_1w_2^{m-1}$, if $m$ is even;

\vspace{-1mm}
\item[(ii)] either (\ref{representantecorrecto}) or the sum of~(\ref{representantecorrecto}) with $u_1w_2^t$, if $m=2t+1$.
\end{itemize}
The actual representative in~(ii) depends on whether the second Bockstein of~(\ref{representantecorrecto}) is trivial or not. [As noted above,~(\ref{representantecorrecto}) alone gives the right representative if $t$ is even; however, the final considerations in Section~\ref{HBring} imply that the extra summand $u_1w_2^t$ is actually needed for odd $t$.] The BSS of $B(\P^m,2)$ collapses from this point on for dimensional reasons. 

\medskip\noindent{\bf Immediate consequences.} The above BSS-analysis has a number of standard implications. First, we see that the torsion-free subgroups in $H^*(B(\P^m,2))$ are as described in Theorem~\ref{aditivoB}, with a torsion-free positive-dimensional cohomology generator, $e_{2m-1}$ for even $m$, and $e_m$ for odd $m$. Their mod~2 reductions are described (partially\footnote{The indeterminacy will be removed in Section~\ref{HBring}.}, for $m\equiv3\bmod4$) in~(i) and ~(ii) above. Second, multiplication by 4 kills the torsion subgroups in the integral cohomology of $B(\P^m,2)$. Next, not only does the map~(\ref{obvious-inclusion-B}) give a surjection on the first page of the corresponding BSS's (Lemma~\ref{chb2pmmod2}), but on positive degrees of the second page level, it maps onto non-permanent cycles. Together with the collapse of both spectral sequences from  their third pages on, this yields the first two items in Corollary~\ref{suprayeccionB}. The last of the immediate consequences of our BSS-analysis for $B(\P^m,2)$ is that we have a good hold on the number of direct summands $\mathbb{Z}_2$ and $\mathbb{Z}_4$ in the integral cohomology of $B(\P^m,2)$. Indeed, these are given by the $\F2$-dimension of the images of the first and second Bockstein differentials, respectively. The explicit counting of dimensions (which yields the proof of Theorem~\ref{aditivoB}) is done in the next paragraphs.

\medskip\noindent{\bf Additive counting.} In view of~(\ref{losmonomios}) and~(\ref{betas2}), an $\F2$-basis for the $\beta_2$-image is given by the monomials $w_2^{2i}$ for $1\leq i\leq m-2+\delta$ where $m=2t+\delta$, $\delta\in\{0,1\}$ (note that the $\beta_2$-indeterminacy inherent in~(ii) above does not play a role here). Therefore, there is a single $\mathbb{Z}_4$-summand only in each positive dimension $n$ satisfying $n<2m-1$ and $\hspace{.25mm}n\equiv0\bmod4$.

\medskip
Counting the $\F2$-dimension of the $\Sq^1$-image gets (combinatorially) more involved, but the task is simplified by working in terms of the auxiliary spectral sequence. Level-$0$ (i.e.~filtered) $\Sq^1$-boundaries have $\F2$-basis given by the monomials in~(\ref{basemod2Bdisp}) satisfying~(\ref{rankd0}); level-$1$ $\Sq^1$-boundaries (i.e.~$d_1$-differentials in the auxiliary spectral sequence) have the $\F2$-basis indicated on the right hand side of the equations in~(\ref{rankd1par}) and~(\ref{rankd1imp}). Since there are no higher-level $\Sq^1$-boundaries (i.e.~higher differentials), we find that, up to elements of higher auxiliary filtration (a proviso which is irrelevant for the purpose of counting $\F2$-dimensions), an $\F2$-basis for the $\Sq^1$-boundaries consists of the monomials $u_1^\varepsilon v_1^rw_2^s$ satisfying one of the following two (disjoint) sets of conditions:
\begin{eqnarray}
&\mbox{$\varepsilon\leq1$, $\;\;r+s<m$, $\;\;r>0$, \ \ and $\;\;\varepsilon+r+s\equiv0\;\bmod2$;}\label{level1}&\\&\mbox{$\varepsilon=\delta$, $\;\;r=m-2i-2+\delta$, \ \ and  $\;\;s=2i+1-\delta$, \ \ for $\;\;\delta\leq i\leq t-1+\delta$.}\label{level2}&
\end{eqnarray}

Theorem~\ref{aditivoB} now follows from a dimension-wise count of the above basis elements. The required checking is straightforward, but the legwork comes from the large number of cases to consider. For the reader's benefit, we illustrate the type of counting needed by working out a representative case, namely the one corresponding to the eighth line in the description of $H^*(B(\P^{2t+1},2))$ in Theorem~\ref{aditivoB}: We want to count the number of basis elements $u_1^\varepsilon v_1^rw_2^s$ satisfying~(\ref{level1}) or~(\ref{level2}), as well as
\begin{equation}\label{lailustracion}
m=2t+1<\dim(u_1^\varepsilon v_1^rw_2^s)=4a+1\leq4t+1.
\end{equation}
Note that the equality in~(\ref{lailustracion}) and the last condition in~(\ref{level1}) force $s$ to be odd, so the equality in~(\ref{lailustracion}) becomes $1\equiv2+r+\varepsilon\bmod4$. This happens only for $r\equiv2\bmod4$ (with $\varepsilon=1$) or $r\equiv3\bmod4$ (with $\varepsilon=0$). Thus, the actual possibilities for the pair $(\varepsilon,r)$ are $(1,4i-2)$ and $(0,4i-1)$---both with $s=2a-2i+1$ in view of the equality in~(\ref{lailustracion})---for $1\leq i\leq t-a$, where the latter inequality comes from the second condition in~(\ref{level1}). Therefore, there are $2(t-a)$ basis elements in dimension $4a+1$ accounted for by~(\ref{level1}). The extra basis element reported by the group $\langle2t+1-2a\rangle$ in Theorem~\ref{aditivoB} comes from~(\ref{level2}), where the dimensional hypothesis in~(\ref{lailustracion}) becomes $i=2a-t$ (this is in the range indicated in~(\ref{level2}), in view of~(\ref{lailustracion})).

\section{The ring structure of $H^*(B(\P^m,2))$}\label{HBring}
We now prove Theorems~\ref{chb2pm} and~\ref{baseB}. Unlike the case of $F(\P^m,2)$, where the proof of Theorem~\ref{chf2pm} uses the auxiliary algebraic model $\mathcal{R}_m$, proofs in this section depend on a very explicit handling of relations in the torsion subgroups of the integral cohomology ring of $B(\P^m,2)$. In particular, the method in the final part of this section (proof of Theorem~\ref{baseB}) is similar to the deduction of the relations $R_{m+s}$ in Lemma~\ref{arithauxi} and their use in Corollary~\ref{basemod2B} for easily obtaining an additive basis for $H^*(B(\P^m,2);\F2)$.

\medskip\noindent
{\bf The relations: simplifying considerations.} The equations in~(\ref{lasrelsevensB}) and~(\ref{lasrelsoddssB}) corresponding to $d_4^{\,t+\delta}=0$ follow from dimensional considerations. This is also the case for the family of equations in~(\ref{lasrelsevensB}) involving $e_{2m-1}$. Further, the first three equations in~(\ref{lasrelsoddssB}) follow respectively from the three relations in~(\ref{comunesB}) because the maps in~(\ref{obvious-inclusion-B}) are compatible under the equatorial inclusion $B(\P^{2t+1},2)\hookrightarrow B(\P^{2t+2},2)$. We now focus on 
\begin{equation}\label{cincorestantes}
\mbox{the three equations in~(\ref{comunesB}) and the first two equations in~(\ref{lasrelsevensB}).}
\end{equation}

A straightforward calculation (left to the reader) using~(\ref{relacionesenterasenD8})--(\ref{losvaloresdelareductionB}),~(\ref{lassigmatoras}), and Lemma~\ref{arithauxi} shows that the equations in~(\ref{cincorestantes}) hold after applying the mod 2 reduction morphism $\rho\colon H^*(\bpm)\to H^*(\bpm,\F2)$. The latter map is monic on torsion elements of dimension not divisible by~$4$ (where there are no copies of $\mathbb{Z}_4$), so that the equations in~(\ref{cincorestantes}) lying in dimensions not divisible by~$4$ already hold in $H^*(\bpm)$. As for the equations in~(\ref{cincorestantes}) that lie in dimensions divisible by $4$, note that:
\begin{itemize}
\item the equatorial inclusion $B(\P^{2t},2)\hookrightarrow B(\P^{2t+1},2)$ induces a cohomology epimorphism in even dimensions (Corollary~\ref{suprayeccionB}.\ref{factepiB}), and 
\item the groups $H^*(B(\P^{2t},2))$ and $H^*(B(\P^{2t+1},2))$ are isomorphic in even dimensions not greater than $4t-1$ (Theorem~\ref{aditivoB}).
\end{itemize}
So, the only equations in~(\ref{cincorestantes}) actually requiring direct verification are
\begin{eqnarray} 
&&a_2\sigma_{2t}=0\mbox{ \ \ and \ \ } b_2\sigma_{2t}+\iota_{2t+2}=0\mbox{ \ \ for $t$ odd, $\;t\geq3$;}\label{AprimeraA}\\ &&\ \hspace{1cm}b_2d_4\sigma_{2t-2}+\iota_{2t+4}=0\mbox{ \ \ for $t$ even, $\;t\geq4$,}\label{AsegundaA}
\end{eqnarray}
all of these with $\delta=0$ (i.e.~as elements of $H^*(B(\P^{2t},2))$), as well as
\begin{eqnarray} 
&&a_2\sigma_{2t}=0\mbox{ \ \ and \ \ } b_2\sigma_{2t}+\iota_{2t+2}=0\mbox{ \ \ for $t=1$;}\label{BprimeraB}\\ &&\ \hspace{1cm}b_2d_4\sigma_{2t-2}+\iota_{2t+4}=0\mbox{ \ \ for $t=2$,}\label{BsegundaB}
\end{eqnarray}
all of these with $\delta=1$ (i.e.~as elements of $H^*(B(\P^{2t+1},2)))$.

\medskip\noindent
{\bf The relations: strategy of proof.} Equations (\ref{AprimeraA})--(\ref{BsegundaB}) can be approached\footnote{The idea can be used to verify most of the relations claimed in Theorem~\ref{chb2pm} (e.g.~all of the equations in~(\ref{cincorestantes})---except for the second equation in~(\ref{BprimeraB}), see below), but the legwork is conveniently reduced by the above `simplifying considerations'.} through the commutative diagram
\begin{equation}\label{strategia}
\begin{diagram}
\node{H^{*-1}(BD_8)}\arrow{s,t}{\beta^*}\arrow{e,t}{\rho}\node{H^{*-1}(BD_8;\F2)}\arrow{s,t}{\beta^*}\arrow{e,t}{\partial}\node{H^{*}(BD_8)}\arrow{s,t}{\beta^*}\\\node{H^{*-1}(B(\P^m,2))}\arrow{e,t}{\rho}\node{H^{*-1}(B(\P^m,2);\F2)}\arrow{e,t}{\partial}\node{H^{*}(B(\P^m,2))}
\end{diagram}
\end{equation}
where the rows are portions of the long exact sequences giving the corresponding BSS's. Namely, exactness implies that the triviality of an element $\zeta\in H^{*}(B(\P^m,2))$ with $2\zeta=0$---i.e.~in the image of the boundary operator of the bottom row---is established by showing that $\zeta$ lies in the image of the composite lower row. Such a task can be carried out in terms of the composite top row: it suffices to find elements $\xi\in H^{*-1}(BD_8)$ and $\eta\in H^{*-1}(BD_8;\F2)$ with
\begin{equation}\label{thepointahacer}
\rho(\xi)\equiv\eta\;\bmod\mbox{Ker}(\beta^*)\qquad\mbox{and}\qquad\beta^*(\partial(\eta))=\zeta.
\end{equation}
The point is that the top row in~(\ref{strategia}) is fully accessible in view of~(\ref{relacionesenterasenD8})--(\ref{losvaloresdelareductionB}) and the fact (Lemma~\ref{handelauxi} below) that the connecting morphism
\begin{equation}\label{connectingtop}
\partial\colon H^{*-1}(BD_8;\F2)\to H^*(BD_8)
\end{equation}
is well understood in terms of the Wall-Hamada resolution for the trivial $D_8$-module $\mathbb{Z}$. The relevant information can be found in~\cite{handeltohoku} (see particularly Proposition~4.3, Equation~(5.1), and the proofs of Theorems~5.2 and~5.5), where a fairly complete description of the multiplicative properties of the cohomology of $D_8$ is presented in great detail. The explicit result we need is:

\begin{lema}[\cite{handeltohoku}]\label{handelauxi}
The connecting map in~{\em(\ref{connectingtop})} is characterized by $$\partial(u_1^\varepsilon v_1^{2i_1+\varepsilon_1}w_2^{2i_2+\varepsilon_2})=\begin{cases}\varepsilon a_2^{i_1}b_2d_4^{i_2},&\varepsilon_1=\varepsilon_2=0;\\\varepsilon a_2^{i_1}b_2c_3d_4^{i_2},&\varepsilon_1=\varepsilon_2=1;\\(1+\varepsilon)a_2^{i_1+1}d_4^{i_2},&\varepsilon_1=1\mbox{ and }\varepsilon_2=0;\\(1+\varepsilon)a_2^{i_1}c_3^{1-\varepsilon}d_4^{i_2+\varepsilon},&\varepsilon_1=0\mbox{ and }\varepsilon_2=1,\end{cases}$$ for integers $\varepsilon,\varepsilon_1,\varepsilon_2\in\{0,1\}$ and $i_1,i_2\geq0$.
\end{lema}
Note that $\partial(u_1v_1^{2i_1}w_2^{2i_2+1})=0$ for $i_1>0$, but $\partial(u_1w_2^{2i_2+1})=2d_4^{i_2+1}$. This behavior leads to the summands ``$\iota_{2t+2}$'' and ``$\iota_{2t+4}$'' in~(\ref{AprimeraA})--(\ref{BsegundaB}).

\medskip\noindent
{\bf The relations: main computation instructions.} Elements satisfying~(\ref{thepointahacer}) can be chosen as follows:
\begin{itemize}
\item For $\zeta=a_2\sigma_{2t}\,$ with $t=2\ell+1$, $\ell\geq1$, and $\delta=0$, take $$\eta=\sum_{j=0}^{\ell}\binom{2t-2j}{2j}v_1^{2t+1-4j}w_2^{2j}\quad\mbox{ and }\quad\xi=\sum_{j=0}^{\ell}\binom{2t-1-2j}{2j+1}a_2^{t-1-2j}c_3d_4^j.$$ 
The term in Ker$(\beta^*)$ needed in~(\ref{thepointahacer}) is the $v_1$-multiple of the first sum in~(\ref{lassigmatoras}).
\item For $\zeta=b_2\sigma_{2t}+\iota_{2t+2}\,$ with $t=2\ell+1$, $\ell\geq1$, and $\delta=0$, take $$\eta=u_1w_2^t+\sum_{j=0}^{\ell}\binom{2t-2j}{2j}u_1v_1^{2t-4j}w_2^{2j}\quad\mbox{ and }\quad\xi=\sum_{j=0}^{\ell-1}\binom{2t-1-2j}{2j+1}a_2^{t-2-2j}b_2c_3d_4^j.$$ 
The term in Ker$(\beta^*)$ needed in~(\ref{thepointahacer}) is the $u_1$-multiple of the first sum in~(\ref{lassigmatoras}).
\item For $\zeta=b_2d_4\sigma_{2t-2}+\iota_{2t+4}$ with $t=2\ell$, $\ell\geq2$, and $\delta=0$, take $$\eta=u_1w_2^{t+1}+\sum_{j=0}^{\ell-1}\binom{2t-2-2j}{2j}u_1v_1^{2t-2-4j}w_2^{2+2j}$$ and $$\xi=\sum_{j=0}^{\ell-2}\binom{2t-3-2j}{2j+1}a_2^{t-3-2j}b_2c_3d_4^{j+1}.$$ 
The term in Ker$(\beta^*)$ needed in~(\ref{thepointahacer}) is the $u_1$-multiple of $R_{m+2}$ in~Lemma~\ref{arithauxi}.
\item For $\zeta=a_2\sigma_{2t}$ with $t=\delta=1$, take $\eta=v_1^3$ and $\xi=0$. 
The term in Ker$(\beta^*)$ needed in~(\ref{thepointahacer}) is the first sum in~(\ref{lassigmatoras}).
\item For $\zeta=b_2d_4\sigma_{2t-2}+\iota_{2t+4}$ with $t=2$ and $\delta=1$, take $\eta=u_1v_1^2w_2^2+u_1w_2^3$ and $\xi=a_2b_2c_3$. 
The term in Ker$(\beta^*)$ needed in~(\ref{thepointahacer}) is the $u_1$-multiple of $R_{m+1}$ in~Lemma~\ref{arithauxi}.
\end{itemize}

The second equation in~(\ref{BprimeraB})---the only equation among those involving only torsion elements, and that we have not yet indicated how to check---is exceptional: the method fails to verify it  because the left-most vertical map in~(\ref{strategia}) is not surjective (we deal below with this case). Indeed, although the $\partial$-image of the element
\begin{equation}\label{reduccionentera3}
u_1v_1^2+u_1w_2\in H^3(B(\P^3,2);\F2)
\end{equation}
is $b_2\sigma_2+\iota_4\in H^*(B(\P^m,2))$---the element asserted to be trivial---, any $\rho$-preimage of~(\ref{reduccionentera3}) involves the torsion-free class $e_3$, an element not in the image of the right-most vertical map in~(\ref{strategia}). To clarify this, note that~(\ref{detmncn}), Lemma~\ref{laobsdehand}, Theorem~9.1 in~\cite{taylor}, and the fact (coming from Theorem~\ref{aditivoB}) that $H^4(B(\P^3,2))$ is a cyclic group of order $4$ (necessarily generated by $d_4$) imply the relation $b_2^2=2d_4$ in this group. This is the second equation in~(\ref{BprimeraB}) in view of~(\ref{relacionesenterasenD8}), thus completing the verification of~(\ref{comunesB})--(\ref{lasrelsoddssB}). But more importantly, the new information can be used to shed light on the above viewpoint. Namely, exactness of the bottom row in~(\ref{strategia}) implies that the element in~(\ref{reduccionentera3}) does lie in the image of the mod~2 reduction map $\rho\colon H^3(B(\P^3,2))\to H^3(B(\P^3,2);\F2)$. But $H^3(B(\P^3,2))=\mathbb{Z}\oplus\mathbb{Z}_2$, where $c_3$---the generator of the torsion subgroup---has $\rho(c_3)=v_1w_2$ in view of~(\ref{losvaloresdelareductionB}). Since an $\F2$-basis for $H^3(B(\P^3,2);\F2)$ is given by the three elements $u_1v_1^2$, $v_1w_2$, and $u_1w_2$ (Corollary~\ref{basemod2B}), the torsion-free class $e_3$ in Theorem~\ref{chb2pm}(b) can actually be chosen to have~(\ref{reduccionentera3}) as its mod~2 reduction. In particular, since $H^5(B(\P^3,2))=\mathbb{Z}_2$ (so that the mod~2 reduction map $\rho\colon H^*(B(\P^3,2))\to H^*(B(\P^3,2);\F2)$ is injective in dimension 5) and $H^k(B(\P^3,2))=0$ for $k\geq6$, the relations in~(\ref{lasrelsoddssBenteras}) are easily proved for $m=3$ by checking them after applying the mod~2 reduction map. 

\medskip
The same idea will be used below to verify the equations in~(\ref{lasrelsoddssBenteras}) for general (odd) $m$---the only relations we have not yet indicated how to verify. As for $m=3$, the explicit calculations require making a choice for the integral classes in Theorem~\ref{chb2pm}, which in turn depends on a description of minimal additive generators for $H^*(\bpm)$---a task whose solution we explain next.

\medskip\noindent
{\bf Minimal additive generators and ring presentation.} For $0\leq s\leq r$ let $R_{r,s}$ stand for the element $\sum_{i\geq0}\binom{r-s-i}ia_2^{r-s-2i}d_4^{s+i}\in H^{2r+2s}(BD_8)$ as well as its image under the map $\beta: B(\P^m, 2) \to BD_8$ in~(\ref{obvious-inclusion-B}). There are identities
\begin{equation}\label{identificacion}
R_{r,0}=\sigma_{2r},\quad R_{r,1}=\sigma_{2r+2}-a_2\sigma_{2r},\quad\mbox{and}\quad R_{r,s+2}=d_4R_{r,s}-a_2R_{r,s+1}
\end{equation}
where the first one holds by definition, and the last two are based on the binomial identity $\binom{a}{b}=\binom{a+1}{b+1}-\binom{a}{b+1}$. The next result uses the elements $\iota_{2r}$ in Theorem~\ref{chb2pm}.

\begin{lema}\label{reltionskdanbaseB}
Let $m=2t+\delta$, $\delta\in\{0,1\}$. The following elements vanish in \emph{$H^*(B(\P^m,2))$:}
{\em \begin{enumerate}
\item\label{nadauxiliar} {\em $a_2R_{t,s}\;$ and $\,\;b_2R_{t,s}+\iota_{2t+2s+2},\;$ for $\;0\leq s\leq t$;}
\item {\em $c_3R_{t-1+\delta,s}\;$ for $\;0\leq s\leq t-1+\delta$.}
\end{enumerate}}
\end{lema}
\begin{proof}
This is an easy exercise using the relations~(\ref{comunesB})--(\ref{lasrelsoddssB}) and~(\ref{identificacion})---in the case of $a_2R_{t,s}$, note that the $c_3$-multiple of the first equation in~(\ref{lasrelsevensB}) becomes $a_2R_{t,1}=0$ in view of the last equation in~(\ref{relacionesenterasenD8}).
\end{proof}

Let $H^*(m)$ be the subring of $H^*(\bpm)$ generated by the classes $a_2,b_2,c_4,d_4$. Thus, besides the unit $1\in H^0(m)$, $H^*(m)$ consists of all the torsion elements in $H^*(\bpm)$. Alternatively, $H^*(m)$ is the image of the morphism induced by the map in~(\ref{obvious-inclusion-B}). 

\begin{proof}[Proof of Theorem~{\em\ref{baseB}}]
The monomials $a_2^ib_2^\varepsilon c_3^{\varepsilon'}d_4^j$ with $i,j\geq0$ and $\varepsilon,\varepsilon'\in\{0,1\}$ are additive generators for $H^*(BD_8)$ in view of~(\ref{relacionesenterasenD8}). Corollary~\ref{suprayeccionB}, Lemma~\ref{reltionskdanbaseB}, and the relation $d_4^{t+\delta}=0$ in~(\ref{lasrelsevensB}) and~(\ref{lasrelsoddssB}) then imply that the elements in~(\ref{gendoresenterosB}) are additive generators for $H^*(m)$ in positive dimensions. Thus, the proof reduces to checking that the elements in~(\ref{gendoresenterosB}) give the right size for the groups reported in Theorem~\ref{aditivoB}. Such a task requires a dimension-wise count analogous to that of the basis elements $u_1^\varepsilon v_1^rw_2^s$ satisfying~(\ref{level1}) or~(\ref{level2}). [The current counting gives more precise information than the one noted at the end of Section~\ref{ahoralaB} since the latter one is performed on elements capturing integral cohomological information only up to higher auxiliary filtration.] The counting needed now is rather simple, and we omit the straighforward details. Yet, for the reader's convenience, Example~\ref{elexale} below deals with a couple of representative cases, namely the ones corresponding to the $(6+\delta)$-th line in the description of $H^*(B(\P^{2t+\delta},2))$ in Theorem~\ref{aditivoB}.
\end{proof}

\begin{ejemplo}\label{elexale}{\em
In dimensions $4\ell$ with $2t+1<4\ell\leq4t+1$, the monomials in~(\ref{gendoresenterosB}) take either one of the forms $a_2^{2i}d_4^{\ell-i}$ and $a_2^{2i-1}b_2d_4^{\ell-i}$ for $0\leq i\leq t-\ell$. For $i>0$, these give $2(t-\ell)$ elements of order 2, whereas the case $i=0$---giving the element $d^\ell$---accounts for a $\mathbb{Z}_4$-group.
}\end{ejemplo}

The above argument also shows that $H^*(m)$ is presented as a ring as indicated in Theorem~\ref{baseB}, except that one has to remove the relations involving the torsion-free positive-dimensional classes $e_{2m-1}$ (for even $m$) and $e_m$ (for odd $m$).

\begin{proof}[Proof of Theorem~{\em\ref{chb2pm}}---sketch of conclusion] It remains to verify the relations in~(\ref{lasrelsoddssBenteras}), that is, the instructions for multiplying with the torsion-free class $e_m\in H^m(\bpm)$ in Theorem~\ref{chb2pm}(b). As a first step we choose explicit generators for all positive-dimensional $\mathbb{Z}$-groups.

\medskip
As $H^{4t-1}(B(\P^{2t},2))=\mathbb{Z}$, there is no real choice to make (except for sign) for $m=2t$: Corollary~\ref{basemod2B} forces
\begin{equation}\label{mod2redpar}
\rho(e_{2m-1})=u_1w_2^{m-1},
\end{equation}
which is the only nonzero element in $H^{4t-1}(B(\P^{2t},2);\F2)=\mathbb{Z}_2$---(\ref{mod2redpar}) has also been noted at the end of the paragraph `Second order Bocksteins' in Section~\ref{ahoralaB}. 

\medskip
The situation for $m$ odd is not as direct, but can still be analyzed using~(\ref{strategia}) with $*=m+1$. Namely, Theorem~\ref{baseB} and Corollary~\ref{basemod2B} give explicit minimal generators (actual $\F2$-basis if no $\mathbb{Z}_4$-summands are involved) for the torsion subgroups of the groups in the lower row of~(\ref{strategia}), whereas~(\ref{losvaloresdelareductionB}), Corollary~\ref{suprayeccionB}.\ref{factepiB}, and Lemma~\ref{handelauxi} can be used to describe the morphisms between these groups. The morphisms behave transparently on bases, sending basis elements to zero or to other basis elements, except for the basis element $u_1v_1^{m-1}\in H^m(\bpm;\F2)$. Indeed, the second relation in~(\ref{comunesB}) is needed to express $\partial(u_1v_1^{m-1})$ as a linear combination of minimal generators in $H^{m+1}(\bpm)$. This yields detailed $\F2$-bases for the kernel of $\partial$ and for the image under $\rho$ of the torsion subgroup of $H^m(\bpm)$ and, as a result, an element is singled out in the former kernel-group which is not in the latter image-group. Then, just as in the case $m=3$ discussed right after~(\ref{reduccionentera3}), exactness of the lower row in~(\ref{strategia}) implies that the singled-out element must be the mod~2 reduction of a torsion-free class $e_m$. The reader is encouraged to fill in the easy details verifying the above discussion, and we content ourselves with reporting the net outcome: For $m=2t+1$, the class $e_m$ in Theorem~\ref{chb2pm}(b) can be chosen to have
\begin{equation}\label{red1yeah}
\rho(e_m)=\sum_{i\geq0}\binom{t-i}{i}u_1v_1^{2t-4i}w_2^{2i}+t\hspace{.3mm}u_1w_2^t.
\end{equation}
Note this is in agreement---and refines---the considerations at the end of the paragraph `Second order Bocksteins' in Section~\ref{ahoralaB}.

\medskip
The remainder of the proof is standard: The first relation in~(\ref{lasrelsoddssBenteras}), as well as the last two for $m\leq3$, hold for dimensional reasons (note that the sum in~(\ref{lasrelsoddssBenteras}) is empty if $m\leq3$, whereas the relation $0=b_2R_{1,1}+\iota_{6}$ in Lemma~\ref{reltionskdanbaseB}.\ref{nadauxiliar} gives the triviality of $b_2d_4$, the right-hand-side term in the third relation in~(\ref{lasrelsoddssBenteras}) for $m=3$). For the rest of the relations one first shows, by straightforward calculation (see Example~\ref{ejemplito1} below), that they hold after evaluating under the mod~2 reduction map $\rho\colon H^*(\bpm)\to H^*(\bpm;\F2)$. As this map is monic on torsion elements of dimension not divisible by four, the asserted relations in $H^*(\bpm)$ hold for free, except for the third relation in~(\ref{lasrelsoddssBenteras}) if $m\equiv1\bmod4$. Indeed, if $\ell\geq1$, the kernel of $\rho\colon H^{4\ell+4}(B(\P^{4\ell+1},2))\to H^{4\ell+4}(B(\P^{4\ell+1},2);\F2)$ is a copy of $\mathbb{Z}_2$ generated by $2d_4^{\,\ell+1}$, so that all we have here is $c_3e_{4\ell+1}=\eta d_4^{\,\ell+1}$ for $\eta\in\{0,2\}$. To solve the indeterminacy (for $m\neq5$), compute $$\eta d_4^{\,\ell+2}=c_3(d_4e_{4\ell+1})=c_3\sum_{i=1}^{\ell}\binom{2\ell-i}{i-1}a_2^{2\ell-2i}b_2c_3d_4^i=\sum_{i=1}^{\ell}\binom{2\ell-i}{i-1}a_2^{2\ell+1-2i}b_2d_4^{i+1}$$ and note that the last sum is the $d_4$-multiple of the left-hand-side term of the relation $b_2R_{2\ell,1}=\iota_{4\ell+4}$ in Lemma~\ref{reltionskdanbaseB}.\ref{nadauxiliar}. This yields $\eta d_4^{\,\ell+2}=2d_4^{\,\ell+2}$ or, equivalently (as $d_4^{\,\ell+2}$~is of order~4 if $\ell\geq2$), $\eta=2$.
\end{proof}

\begin{ejemplo}\label{ejemplito1}{\em
We verify in detail the last relation in~(\ref{lasrelsoddssBenteras}). Recall $m=2t+1$ and $t=2\ell+\kappa$ with $\kappa\in\{0,1\}$. Use~(\ref{losvaloresdelareductionB}),~(\ref{red1yeah}), and Lemma~\ref{arithauxi} (with $s=1$) to get
\begin{eqnarray*}
\rho(d_4e_m)&=&\sum_{i=0}^{\ell}\binom{t-i}{i}u_1v_1^{2t-4i}w_2^{2i+2}+tu_1w_2^{t+2}\\ &=&u_1v_1^{2t}w_2^2+\sum_{i=1}^{\ell}\binom{t-i}{i}u_1v_1^{2t-4i}w_2^{2i+2}+tu_1w_2^{t+2}\\ &=&u_1w_2\sum_{i=1}^{2\ell+\kappa}\binom{2t-i}{i}v_1^{2t-2i}w_2^{i+1}+\sum_{i=1}^{\ell}\binom{t{-}i}{i}u_1v_1^{2t-4i}w_2^{2i+2}+tu_1w_2^{t+2}.
\end{eqnarray*}
Note that the even indices $i$ in the summation from $1$ to $2\ell+\kappa$ cancel out the summation running over $1\leq i\leq\ell$. On the other hand, if $t$ is odd (i.e.~if $\kappa=1$), then the summand with index $i=2\ell+1$ cancels out the final summand $tu_1w_2^{t+2}$---if $\kappa=0$, none of these terms appear. The above expression then simplifies to $$\rho(d_4e_m)=u_1w_2\sum_{i=1}^{\ell}\binom{2t-2i+1}{2i-1}v_1^{2t-4i+2}w_2^{2i}=\sum_{i=1}^{\ell}\binom{t-i}{i-1}u_1v_1^{2t-4i+2}w_2^{2i+1}.$$ But~(\ref{losvaloresdelareductionB}) implies $u_1v_1^{2t-4i+2}w_2^{2i+1}=v_1^{2t-4i}\cdot u_1v_1\cdot v_1w_2\cdot w_2^{2i}=\rho(a_2^{t-2i}b_2c_3d_4^{\,i})$ which, as explained in the proof sketch above, gives the $d_4$-relation asserted in~(\ref{lasrelsoddssBenteras}).
}\end{ejemplo}

\vspace{1.5mm}
Carlos Dom\'{\i}nguez\quad {\tt cda@math.cinvestav.mx}

{\sl Departamento de Matem\'aticas, CINVESTAV--IPN

M\'exico City 07000, M\'exico}

\vspace{2mm}

Jes\'us Gonz\'alez\quad {\tt jesus@math.cinvestav.mx}

{\sl Departamento de Matem\'aticas, CINVESTAV--IPN

M\'exico City 07000, M\'exico}

\vspace{2mm}

Peter Landweber\quad {\tt landwebe@math.rutgers.edu}

{\sl Department of Mathematics, Rutgers University

Piscataway, NJ 08854, USA}

\end{document}